\documentclass[a4paper, 12pt,oneside,reqno]{amsart}
\usepackage[a4paper]{geometry}
\usepackage[english]{babel}
\usepackage[T1]{fontenc}
\usepackage{amsfonts}
\usepackage{mathrsfs}
\usepackage{tikz}
\usetikzlibrary{arrows,shapes,snakes,automata,backgrounds,petri,through,positioning}
\usetikzlibrary{intersections}

\usepackage[matrix,arrow]{xy}
\usepackage{amssymb,amscd,amsthm,amsmath}
\usepackage[a4paper]{geometry}
\usepackage{amsmath}
\usepackage{amssymb}
\usepackage{amsthm}
\usepackage[colorlinks=true, allcolors=blue]{hyperref}

\newtheorem{theorem}{Theorem}[section]
\newtheorem{lemma}[theorem]{Lemma}

\newtheorem{proposition}[theorem]{Proposition}
\newtheorem{corollary}[theorem]{Corollary}
\theoremstyle{definition}
\newtheorem{definition}[theorem]{Definition}
\newtheorem{example}[theorem]{Example}
\newtheorem{remark}[theorem]{Remark}

\renewcommand{\phi}{\varphi}

\title{On Logarithmic Poisson and De Rham Cohomology Groups of a Class of  Inhomogeneous Divisors.}
\author{Kamtila  Kari}
\address{ University of Maroua}
\email{kamtilakari@gmail.com}
\author{Dongho Joseph}
\address{ University of Maroua}
\email{josephdongho@yahoo.fr}
\author{Diekouam Fotso Luc Emery}
\address{ University of Maroua}
\email{lucdiekouam@yahoo.fr}
\thanks{ }   
\begin{document}

\begin{abstract}
We study logarithmic Poisson and de Rham cohomologies associated with a class of inhomogeneous free divisors in the affine plane. These divisors arise as inhomogeneous deformations of reduced normal crossing divisors and admit explicit bases for their modules of logarithmic vector fields in the sense of Saito. Using these bases, we construct the corresponding logarithmic Poisson structures and describe explicitly the induced Koszul bracket on logarithmic differential 1-forms within the framework of Lie-Rinehart algebras. We then determine the logarithmic Poisson cochain complex and compute its cohomology for the class under consideration. Furthermore, by means of the logarithmic Spencer complex, we identify the corresponding logarithmic de Rham complex and compute its cohomology. These computations provide explicit cohomological invariants for the class of inhomogeneous divisors considered and show how logarithmic Poisson and de Rham theories extend the corresponding constructions for reduced normal crossing divisors.
In addition, we establish an explicit comparison between the logarithmic cohomological theories: they are naturally isomorphic in any degree $k\neq 1$, whereas in degree $1$ the logarithmic de Rham cohomology group is a split one-dimensional extension of the logarithmic Poisson cohomology group.
\end{abstract}
\maketitle
\section{Introduction}
Logarithmic structures associated with divisors provide a natural framework for studying differential and geometric properties near singular loci. A fundamental contribution in this direction is due to \cite{KS}, who introduced the module of logarithmic vector fields along a divisor and established basic structural results for free divisors. In particular, when a divisor is free, the module of logarithmic derivations is locally free and admits a basis whose dual basis gives a corresponding system of logarithmic differential forms. These structures have proved useful in the study of singularities, differential equations, arrangements, and logarithmic differential forms.\\
The logarithmic viewpoint has been developed further in several directions. In particular, logarithmic differential forms and logarithmic connections provide a natural way of extending classical differential-geometric constructions across divisors with singularities. Recent developments in logarithmic algebraic geometry, including the study of logarithmic resolutions and logarithmic cotangent bundles, illustrate the continuing importance of this framework; see, for example, \cite{DAMH} and \cite{LDER}. These developments motivate the study of algebraic structures naturally defined on modules of logarithmic vector fields and logarithmic differential forms.\\
Poisson geometry provides another important cohomological framework. Since the work of \cite{AL}, Poisson cohomology has been used to study intrinsic properties of Poisson structures and, in particular, their infinitesimal deformations. For a Poisson bivector $\pi$, the Lichnerowicz differential
$d_{\pi}=[\pi,\cdot]_{\mathrm{SN}}$
defines a cochain complex of multivector fields, whose cohomology measures, among other phenomena, infinitesimal symmetries and deformation directions of the Poisson structure. Explicit computations of Poisson cohomology remain an active subject. For instance, \cite{DHFZ} computed Poisson cohomology for linear Poisson structures associated with three-dimensional Lie algebras, while \cite{ARA} further illustrates the role of Poisson cohomology in deformation problems.\\
The interaction between Poisson geometry and inhomogeneous structures is also of special relevance here. \cite{XHMC} studied multidimensional nonhomogeneous quasi-linear systems and their Hamiltonian structures, emphasizing the role of Poisson cohomology in the analysis of admissible Hamiltonian operators. Their work illustrates the usefulness of cohomological methods when a geometric or algebraic structure contains both homogeneous and nonhomogeneous components.\\
In the present work, we combine these two perspectives by studying logarithmic Poisson and logarithmic de Rham cohomologies associated with a class of inhomogeneous divisors in the affine plane. More precisely, we consider divisors of the form
$D =\{h=0\}$, $h = \overset{n}{\underset{i=1}{\sum}} \alpha_{i}(\overset{p}{\underset{j=1}{\Pi}} x_{j})^{n_{i}} =0$; $\alpha_{i}\neq 0, n_{i}>0, n_{i}\neq n_{j}$.
 In particular, in dimension two
$h=\overset{n}{\underset{i=1}{\sum}} \alpha_{i} x^{i} y^{i} $ with $\alpha_1\alpha_n\neq 0$, 
 can be viewed as polynomial deformations of $h_0=xy$. Thus, the homogeneous normal crossing case is recovered as the undeformed member of the family, whereas the additional higher-order terms produce an inhomogeneous divisor.\\
Our motivation comes from the fact that, although logarithmic structures for normal crossing and other well-understood free divisors are classical, explicit cohomological computations for inhomogeneous free divisors appear to be less developed than in several homogeneous settings. In particular, the passage from a homogeneous normal crossing divisor to an inhomogeneous deformation changes the algebraic expressions of the logarithmic vector fields and logarithmic differential forms, while preserving enough of the underlying structure to make explicit calculations possible.\\
The first ingredient of our approach is therefore the construction of Saito-type bases for the module
$Der_{\mathcal{A}}(log D)
=
\{\delta\in Der_{\mathcal{A}} \mid \delta(h)\in h\mathcal{A}\}$, $\mathcal{A}=\mathbb{C}[x_{1},...,x_{p}]$.
The corresponding dual basis gives a convenient description of
$\Omega_\mathcal{A}^1(log D)$,
and allows us to express the logarithmic Poisson structure intrinsically in terms of the logarithmic modules along the ideal $\mathcal{I}=h\mathcal{A}$, $\mathcal{A}=\mathbb{C}[x,y]$. In particular, the logarithmic Hamiltonian map identifies the Poisson-theoretic structure with the Lie-Rinehart algebra structure carried by $Der_{\mathcal{A}}(log \mathcal{I})$. The corresponding Koszul bracket on logarithmic differential forms can then be written explicitly with respect to the Saito basis.\\
This construction naturally leads to a logarithmic version of the Lichnerowicz complex. Rather than considering all polynomial multivector fields, we restrict the Poisson differential to the logarithmic multivector fields generated by $Der_{\mathcal{A}}(log \mathcal{I})$. This produces the logarithmic Poisson complex
$\left(
\wedge_\mathcal{A}^\bullet Der_{\mathcal{A}}(log \mathcal{I}),
d_{\tilde{H}}^\bullet
\right)$,
where $\tilde{H}$ is the logarithmic Hamiltonian map, and its cohomology will be denoted by
$H^\bullet_{log}(\widetilde{P})$.
The use of a Saito basis makes it possible to determine the differentials of this complex explicitly and to compute the corresponding cohomology groups for the class of inhomogeneous divisors under consideration.\\
A second cohomological object arises from the dual logarithmic complex
$\left(
\Omega_\mathcal{A}^\bullet(log \mathcal{I}), \epsilon_\bullet^*
\right)$.
The relation between logarithmic differential forms, logarithmic vector fields, and the Spencer complex provides a natural algebraic realization of the corresponding logarithmic de Rham cohomology. In this way, the two cohomological theories are constructed from dual logarithmic structures associated with the same divisor.\\
After \cite{DJ1} on the logarithmic Poisson cohomology of some homogeneous divisors, our previous work \cite{KK} considered the logarithmic Poisson and logarithmic de Rham cohomological structures associated with the reduced normal crossing divisor $D_0=\{h_0=0\}$. The present paper goes beyond that setting by treating a class of inhomogeneous divisors which deform the normal crossing configuration. Consequently, the calculations obtained here may be viewed as an extension of the normal crossing case to a family in which higher-order terms modify the divisor and the associated logarithmic structures.\\
The main questions addressed in this paper are therefore the following. How do Saito-type bases behave under these inhomogeneous deformations? How can the associated logarithmic Hamiltonian structure and Koszul bracket be expressed explicitly? What are the logarithmic Poisson cohomology groups of the resulting Poisson structures? How are these groups related to the logarithmic de Rham cohomology of the same divisor? Finally, to what extent do the resulting cohomological structures reflect the deformation from the normal crossing divisor $D_0$?\\
We answer these questions by first constructing explicit bases for the logarithmic vector fields and their dual logarithmic differential forms. We then determine the induced logarithmic Poisson structure and its Koszul bracket, construct the logarithmic Poisson complex, and compute its cohomology. On the differential-form side, we use the logarithmic Spencer complex to obtain the corresponding de Rham complex and determine its cohomology. The resulting formulas give an explicit description of how the cohomological invariants change from the homogeneous normal crossing divisor to the class of inhomogeneous divisors considered here.\\
The results thus establish a concrete connection between Saito's logarithmic theory, Poisson cohomology, and de Rham cohomology for inhomogeneous divisors. They also show that the normal crossing divisor $D_0$ is not merely a special example, but the natural reference point from which the logarithmic Poisson and de Rham structures of the inhomogeneous family can be understood.\\
The purpose of this paper is to determine explicitly the logarithmic Poisson and logarithmic de Rham cohomology of the class of inhomogeneous free divisors considered above. In dimension two, we take
$h=\sum_{i=1}^{n}\alpha_i t^i$, $t=xy$, with $\alpha_1\alpha_n\neq0$, $h=th_t$ and $\gcd(h_t,h_t')=1$ in $\mathbb C[t]$. The Saito bases, the logarithmic Hamiltonian structure, and the Koszul bracket are determined explicitly, and the corresponding logarithmic Poisson and de Rham cohomology groups are computed.\\
The main comparison is degree-by-degree: the explicit calculations give fundamental relations
$H_{DR-log}^{1}\simeq H_{log}^{1}(\widetilde{\mathcal P})\oplus\mathbb Ct^{n-1}$ and
$H_{log}^{k}(\widetilde{\mathcal P})\simeq H_{DR-log}^{k}$ for $k\neq1$. We also give an explicit cochain-level comparison for the formal deformation of the normal crossing divisor $D_0=\{xy=0\}$.\\
The paper is organized as follows.
Section 2 collects the necessary preliminaries.
 Section 3 is devoted to Saito bases for the class of inhomogeneous divisors considered here. In Section 4, we determine the associated logarithmic Koszul bracket.
Section 5 contains the explicit computations of the logarithmic Poisson and de Rham
cohomology groups.

\section{Preliminaries}\label{Rapp}
Let $X=\mathbb{C}^{p}$ be the $p$-dimensional affine space with coordinates $(x_1,\ldots,x_p)$.
In the following we consider the polynomial algebra given by $\mathcal{A} = \mathbb{C}[x_{1},...,x_{p}]$.
\begin{definition} \cite{ER}
A Poisson algebra is an $\mathbb{F}$-vector space $\mathcal{A}$ equipped with two bilinear products denoted by $(f,g)\mapsto f.g$ and $(f,g)\mapsto \{f,g\}$ such that:
 \begin{enumerate}
 \item[i.] $(\mathcal{A},.)$ is an associative commutative algebra over $\mathbb{F}$,
 \item[ii.] $(\mathcal{A},\{.,.\})$ is a Lie algebra over $\mathbb{F}$,
 \item[iii.] The two bilinear products satisfy for all $a,b,c\in \mathcal{A}$; the Leibniz rule,
 \begin{equation}\label{R2}
\{a.b, c\}=a.\{b,c\}+b.\{a,c\}.
 \end{equation}
 \end{enumerate}
\end{definition}
\begin{example} Let $\mathcal{A} = \mathbb{F}[x,y]$ and
$\{f, g\}_{0} = \{x,y\}_{0} \left(\dfrac{\partial f}{\partial x } \dfrac{\partial g}{\partial y } - \dfrac{\partial f}{\partial y } \dfrac{\partial g}{\partial x }\right)$ for all $f,g\in \mathcal{A}$. Thus $\mathcal{P}=(\mathcal{A},\{.,.\}_0)$ is a Poisson algebra.
\end{example}

\begin{definition} \cite{KK}
Let $(\mathcal{A},.)$ be an algebra and $\mathbb{F}$ the field of characteristic zero. $Der_{\mathbb{F}}(\mathcal{A})$ is the $\mathcal{A}$-module of derivations of $\mathcal{A}$; that is, $\delta \in Der_{\mathbb{F}}(\mathcal{A})$ if $\delta$ is a linear map $\delta: \mathcal{A} \longrightarrow \mathcal{A}$ satisfying $\delta(a.b)=a.\delta(b)+ \delta(a).b$ for all $a,b\in \mathcal{A}$.
\end{definition}

Denote by $Der_{\mathcal{A}}$ the set of all derivations on $\mathcal{A} = \mathbb{C}[x_{1},\dots,x_{p}]$ and by $\Omega_{\mathcal{A}}$ its dual.
\begin{definition} \cite{MNN} A Lie-Rinehart algebra is a triple $(\mathcal{A},[.,.], \rho)$ where $[.,.]$ is the Lie bracket and $\rho: \Omega_\mathcal{A}\longrightarrow Der_\mathcal{A}$ is simultaneously a morphism of $\mathcal{A}$-modules and of $\mathbb{F}$-Lie algebras satisfying $[x,a.y] = \rho(x)(a).y + a.[x,y]$ for any $a\in\mathcal{A}$, $x,y\in \Omega_\mathcal{A}$.
\end{definition}
According to \cite{KK}, the Leibniz rule gives, for each $a\in \mathcal{A}$, the map $ad_{a} :\mathcal{A}\rightarrow \mathcal{A}$, $b\longmapsto \{a,b\}$. Moreover for all $a,b\in \mathcal{A}$ we have,
 \begin{eqnarray}
 ad_{ab}(x) = a.\{b,x\}+ b.\{a,x\}=a ad_{b}(x)+b ad_{a}(x).
 \end{eqnarray}
 The map $ad :\mathcal{A}\rightarrow Der_{\mathcal{A}}, a\longmapsto \{a,.\}$ is a derivation of $\mathcal{A}$ with values in $Der_{\mathcal{A}}$. By the universal property of $(\Omega_{\mathcal{A}},d)$, the derivation $ad:\mathcal{A}\rightarrow Der_{\mathcal{A}}$ induces an $\mathcal{A}$-module homomorphism, called the Hamiltonian map defined as follows
\begin{equation} \label{r2}
\begin{array}{clclcl}
 H: & \Omega_{\mathcal{A}} & \longrightarrow & Der_{\mathcal{A}}\\
 & df & \longmapsto & \{f,.\}.\\
 \end{array}
\end{equation}
 such that the following diagram commute,
 \begin{eqnarray}
 \xymatrix{\mathcal{A}\ar[r]^d\ar[dr]_{ad}&\Omega_\mathcal{A}\ar[d]^H\\
         &    Der_\mathcal{A}}
 \end{eqnarray}
 We recall that a Lie-Rinehart-Poisson algebra is a Lie-Rinehart algebra $(\mathcal{A},[.,.], \rho)$ where $\rho:=H$ and $[.,.]$ satisfy the Leibniz rule.
 According to \cite{IV}, every Poisson structure $\{.,.\}$ induces a Poisson bivector $\pi$ satisfying for any $f,g\in \mathcal{A}$,
 \begin{equation}
 \{f,g\}= \pi(df,dg) = \pi^{ij}\frac{\partial f}{\partial x_i}\frac{\partial g}{\partial x_j}.
 \end{equation}
 Note that $\{x_{i},y_{i}\} = \pi^{ij} = -\pi^{ji}$. Let $\sharp : \Omega_\mathcal{A} \longrightarrow Der_\mathcal{A}$, $ \sharp \alpha :=\alpha^{\sharp}$ and define for all $\alpha, \beta \in \Omega_\mathcal{A}$,
\begin{equation} \label{r5}
\beta(\alpha^{\sharp}) = \pi(\alpha, \beta).
\end{equation}
 \begin{definition}\cite{IV}\label{d3}
 The Koszul bracket of two 1-forms is defined on $\Omega_\mathcal{A}$ as follows\\
 $[-,-]_{\Omega_\mathcal{A}} : \Omega_\mathcal{A} \times \Omega_\mathcal{A} \longrightarrow \Omega_\mathcal{A}$ such that
 $[\alpha, \beta]_{\Omega_\mathcal{A}}=\mathcal{L}_{\alpha^{\sharp}}(\beta)-\mathcal{L}_{\beta^{\sharp}}(\alpha)-d\pi(\alpha, \beta)$, for all $\alpha, \beta \in \Omega_{\mathcal{A}}$.
 \end{definition}
If $(dx_{1},...,dx_{p})$ is the basis of $\Omega_{\mathcal{A}}$, we remember that $(dx_{i})^{\sharp} = H(dx_{i})= \pi(dx_{i},-)$.
Consider the action of vector field $Y$ on $f\in \mathcal{A}$ by $\mathcal{L}_{Y}f=Y.f=\langle df, Y\rangle$ which is the Lie derivative. For for all $\alpha\in \Omega_{\mathcal{A}}^q$, and $Y_{1},...,Y_{q-1}\in \wedge^q Der_{\mathcal{A}}$, the inner product $i_{Y}\alpha \in \Omega^{q-1}_{\mathcal{A}}$ is given by,
\begin{equation}\label{relation6}
 (i_{Y}\alpha)(Y_{1},...,Y_{q-1})=\alpha(Y,Y_{1},...,Y_{q-1}).
 \end{equation}
 The Lie derivative verifies the following relation,
 \begin{equation}\label{relation7}
 \mathcal L_{aY}\alpha
 =a\mathcal L_Y\alpha+ da\wedge i_{Y}\alpha.
 \end{equation}
 \begin{remark}
 If $\alpha$ is a 1-form as in our case, $i_{Y}\alpha = \alpha(Y)$ is a function. Thus $da\wedge i_{Y}\alpha=\alpha(Y)da$.
 \end{remark}
 \begin{definition} \cite{JD}
The Weyl algebra of order $p$ over $\mathbb{K}$(a field of
characteristic 0) is the unital associative algebra on
$\mathbb{K}$ which has $2p$ generators $x_{1},...,x_{p};
\partial_{1},...,\partial _{p}$ noted by $A_{p}(\mathbb{K}) =
(x_{1},...,x_{p}; \partial_{1},...,\partial _{p})$, satisfying the
following relations,
\begin{enumerate}
 \item[1.] $[x_{i},x_{j}] = [\partial _{i},\partial_{j}]= 0$ for $i \neq j $;
 \item[2.] $[x_{i} , \partial_{j}] = \delta_{ij}$, where $\delta_{ij}$ is the
 Kronecker symbol and $\partial _{i}=\partial _{x_{i}}$.
\end{enumerate}
It is also noted by $A_{p}$ if there is no ambiguity. Particularly $A_{0}(\mathbb{K})=A_{0} = \mathbb{K}$.
\end{definition}
 \begin{definition} \cite{AZ} $\label{Def3}$
 Polynomial $f=\overset{}{\underset{i_{j}\geq 0}{\sum}}f_{i_{1}...i_{p}}x_{1}^{i_{1}}...x_{p}^{i_{p}}\in \mathbb{C}[x_{1},...,x_{p}]$ is said to be quasi-homogeneous if there exist $\omega_{x_{1}},...,\omega_{x_{p}}\in \mathbb{N}^{*}$ and a number $W\in \mathbb{N}^{*}$ such that for every $(i_{1},...,i_{p})\in supp(f)$, it holds $i_{1}\omega_{x_{1}}+...+i_{p}\omega_{x_{p}} = W$ where $supp(f):= \{(i_{1},...,i_{p})\in (\mathbb{N}^{*})^{p}, f_{i_{1}...i_{p}}\neq 0\}$.
  Each $\omega_{x_{i}}$ represents the weight of $x_{i}$ and $W$ is the weight or the homogeneity degree of $f$.
  \end{definition}
 \begin{definition} \cite{PA}
  A divisor $D=\{h=0\}$ is said to be quasi-homogeneous (or weighted homogeneous) of degree $\varpi$ at $x\in X$,
  if $h$ is weighted homogeneous of degree $\varpi$ at $x\in X$.
  \end{definition}
We denote by $Der_{\mathcal{A}}(logD)$ the module of logarithmic vector fields and its dual $\Omega_{\mathcal{A}}^{1}(log D)$ the module of logarithmic differential 1-forms along $D$.
\begin{definition} \cite{KS} Let $D=\{h=0\}$. A Saito basis of $Der_{\mathcal{A}}(logD)$ is system of the vector fields $\mathcal{B} = \{\delta^{1},..., \delta^{p}\}$ such that $det(\delta^{i})_{i=1}^p = u.h$ and $\delta^{i}(h)\in (h)$ where $u\in \mathcal{A}$. $\mathcal{B}$ is said to be free Saito basis of $Der_{\mathcal{A}}(logD)$ if $u$ is unit.
 \end{definition}
 \begin{definition} \cite{KS} A Saito basis of $\Omega_{\mathcal{A}}^{1}(logD)$ is the system of the 1-forms $\mathcal{B} = \{\omega_{1},..., \omega_{p}\}$ such that $\omega_{1}\wedge...\wedge \omega_{p} = \dfrac{unit}{h}dx_{1}\wedge...\wedge dx_{p}$.
 \end{definition}
 \begin{definition} \cite{ADVR} $D$ is said to be a free divisor if and only if the $\mathcal{A}$-module $Der_{\mathcal{A}}(logD)$ (or, equivalent, $\Omega^{1}_{\mathcal{A}}(logD)$) is a locally free $\mathcal{A}-$module.
 \end{definition}
 \begin{theorem} \cite{KS} \label{T1}
 \begin{enumerate}
 \item[1.] $Der_{\mathcal{A}}(logD)$ is said to be free if and only if there exist $p$ elements $\delta^{1},...,\delta^{p}\in Der_{\mathcal{A}}$ with $\delta^{i} = \overset{p}{\underset{i =1}{\sum}}a_{i}^{j}(x)\partial_{x_{j}}$, $j=1,...,p$ such that the determinant $det(a_{i}^{j}(x))_{_{i,j=1,...,p}}$ is a unit multiple of $h$. Then the set of the vector fields $\{\delta^{1},...,\delta^{p}\}$ is a free basis of $Der_{\mathcal{A}}(logD)$.
 \item[2.] $\Omega_{\mathcal{A}}^{1}(logD)$ is the dual of $Der_{\mathcal{A}}(logD)$ if and only if there exist the 1-forms $\omega_{1},..., \omega_{p} $ such that $\omega_{1}\wedge...\wedge \omega_{p} = \dfrac{unit}{h}dx_{1}\wedge...\wedge dx_{p}$ and $\delta^{i}\lrcorner \omega _{j}=\delta_{ij}$. Therefore, a set of 1-forms $\mathcal{B}^{*} = \{\omega_{1},..., \omega_{p}\}$ makes a system of free basis for $\Omega_{\mathcal{A}}^{1}(logD)$.
 \end{enumerate}
 \end{theorem}
 \begin{proposition}\cite{EF} $\label{P3}$
 Suppose that $\mathcal{B} = \{\delta^{1},...,\delta^{p}\}$ forms a basis of $ Der_{\mathcal{A}}(log D)$.
 Therefore $[\delta^{i},\delta^{j}]_{\Omega_{\mathcal{A}}^{1}(logD)} = 0$ for all $i,j\in \{1,...,p\}$ if and only if
 the basis $\mathcal{B}^{*} = \{\omega_{1},...,\omega_{p}\}$ of $ \Omega_{\mathcal{A}}(log D)$
 satisfying $\delta^{i}\lrcorner \omega _{j}=\delta_{ij}$ consists of closed forms.
 \end{proposition}
  \begin{definition} \cite{KK2} \label{R2}
  We name the logarithmic cochain complex of dimension $p$ associated with all alternate $p$-linear map $\tilde{H}: \Omega^{1}_{\mathcal{A}}(log \mathcal{I}) \longrightarrow Der_{\mathcal{A}}(log \mathcal{I})$, the following sequence,
  \begin{displaymath} \label{CC1}
   (\textbf{C}^{*},d^{*}):...\stackrel{d^{i}_{\tilde{H}}}{\longrightarrow}\wedge^{i}Der_{\mathcal{A}}(log \mathcal{I})
     \stackrel{d^{i+1}_{\tilde{H}}}{\longrightarrow} \wedge^{i+1}Der_{\mathcal{A}}(log \mathcal{I}) \stackrel{d^{i+2}_{\tilde{H}}}{\longrightarrow}...
     \end{displaymath}
    where the differential
      $d^{i}_{\tilde{H}}: \mathfrak{L}^{i-1}_{alt}( \Omega^{1}_{\mathcal{A}}(log\mathcal{I}), \mathcal{A})
      \longrightarrow \mathfrak{L}_{alt}^{i}( \Omega^{1}_{\mathcal{A}}(log\mathcal{I}), \mathcal{A})$ is the logarithmic Poisson differential such that
      for all $f \in \mathfrak{L}_{alt}(\Omega^{1}_{\mathcal{A}}(log\mathcal{I}), \mathcal{A})$
      and $ \omega_{1},...,\omega_{p+1}\in \Omega^{1}_{\mathcal{A}}(log\mathcal{I})$,
      \begin{eqnarray*}
      d^{i}_{\tilde{H}}f(\omega_{1},...,\omega_{p+1})
      &=& \sum_{i=1}^{p+1}(-1)^{i-1}\tilde{H}(\omega_{i})f( \omega_{1},...,\hat{\omega}_{i},...,\omega_{p+1})+ \label{*} ~~~~~~~~~~~~~~~~~~ \\
      & &\sum_{1\leq i\leq j \leq p+1} (-1)^{i+j}f([\omega_{i}, \omega_{j}]_{_{\Omega^{^{1}}_{\mathcal{A}}(log\mathcal{I})}},
      \omega_{1},...,\hat{\omega}_{i},
     ...,\hat{\omega}_{j},...,\omega_{p+1}).
      \end{eqnarray*}
  \end{definition}
  The corresponding cohomology is the Lichnerowicz-Poisson cohomology called the logarithmic Poisson cohomology of the logarithmic Poisson algebra $(\mathcal{A}, \{.,.\}, \tilde{H})$
    where $\tilde{H}$ denote the logarithmic Hamiltonian map. For all $a\in \Omega^{1}_{\mathcal{A}}(log \mathcal{I})$, it verifies the following compatibility,
     $$ [\omega_{i},a\omega_{j}]= \tilde{H}(\omega_{i})(a)\omega_{j} + a[\omega_{i},\omega_{j}].$$
 \begin{definition}\cite{KK2}
The
    $(i-1)-{th}$ logarithmic Poisson cohomology group of the logarithmic Poisson algebra $\tilde{\mathcal{P}}=(\mathcal{A};\{.,.\}_{h}, \tilde{H})$ along the ideal $\mathcal{I}=h\mathcal{A}$ is $H_{log}^{i-1}(\tilde{\mathcal{P}}) = \dfrac{Z^{i}(\textbf{C}^{*},d^{*})}{B^{i}(\textbf{C}^{*},d^{*})}$
  where for all $i\geqslant 1$,
   $Z^{i}(\textbf{C}^{*},d^{*}) = Kerd^{i}$ denote the space of $i-$cocycles and $B^{i}(\textbf{C}^{*},d^{*}) = Imd^{i-1}$ is the space of
     $i-$coboundaries.
   \end{definition}
 Let $\mathcal{D}_{X}$ and $\mathcal{O}_{X}$ respectively the sheaf of
 differential operators and the sheaf of regular functions on the algebraic manifold $X=\mathbb{C}^{p}$, $\mathcal{V}_{0}^{D}(\mathcal{D}_{X})$ is the ring of logarithmic differential operator along divisor $D=\{h=0\}$.\\
\begin{definition} \cite{FJCM}
We call the logarithmic Spencer complex, and denote by $\mathcal{S}p^{\bullet}(log D)$, the complex $\left(\mathcal{V}_{0}^{D}(\mathcal{D}_{X}) \otimes_{\mathcal{O}_{X}} \wedge ^{*}Der(log D) ,~\epsilon_{-*}\right)$ given by,\\
$0\longrightarrow \mathcal{V}_{0}^{D}(\mathcal{D}_{X}) \otimes_{\mathcal{O}_{X}} \wedge ^{n}Der(log D) \longrightarrow... \longrightarrow \mathcal{V}_{0}^{D}(\mathcal{D}_{X}) \otimes_{\mathcal{O}_{X}} \wedge ^{1}Der(log D) \longrightarrow \mathcal{V}_{0}^{D}(\mathcal{D}_{X}) $
where $\epsilon_{-p}(P\otimes(\delta^{1}\wedge...\wedge\delta^{p}))= \overset{p}{\underset{i=1}{\sum}}(-1)^{i-1}P\delta^{i}\otimes (\delta^{1}\wedge...\wedge\hat{\delta^{i}} \wedge... \wedge \delta^{p}) +$\\
$ \overset{}{\underset{1 \leq i<j\leq p}{\sum}}(-1)^{i+j}P\otimes([\delta^{i},\delta^{j}]\wedge\delta^{1}\wedge...\wedge \hat{\delta}_{i}\wedge...\wedge \hat{\delta}_{j}\wedge...\wedge \delta^{p})$; $2\leq p \leq n$. Where $\epsilon_{-1}(P\otimes\delta)=P\delta$.
\end{definition}
 For the Spencer algebra $A_{p} = (x_{1},...,x_{p}, \partial_{x_{1}},..., \partial_{x_{p}})$, consider the quotient of $A_{p}$-modules
 $ M^{logh}= \dfrac{A_{p}}{A_{p}Der_{ \mathcal{A}}(log h)}$ and
$ \widetilde{M}^{logh}= \dfrac{A_{p}}{A_{p}\widetilde{Der _{\mathcal{A}}(log h)}}$ where $\widetilde{Der _{\mathcal{A}}}(logh)=\{  \delta+\dfrac{\delta(h)}{h} | \delta \in Der _{\mathcal{A}}(log h) \}$. \\
 \begin{proposition}\cite{FJCJN} \label{proposition4}\\
Let $f\in \mathcal{A}=\mathbb{C}[x,y]$ be a nonzero reduced
polynomial. There exists a natural isomorphism
$\Omega_{\mathcal{A}}^{\bullet}(log
f)\overset{\simeq}{\longrightarrow}
\textbf{R}Hom_{A_{2}}(M^{log(f)}, \mathcal{A})$.
 \end{proposition}
 \begin{theorem} \cite{FJCJN} \label{Theorem4}\\
For any nonzero reduced polynomial $f\in
\mathcal{A}=\mathbb{C}[x,y]$, the complexes
$(\Omega_{\mathcal{A}}^{\bullet}(log f), \epsilon_{\bullet}^{*})$
and $DR(\widetilde{M}^{log(f)}, \mathcal{A})$ are naturally
quasi-isomorphic.
 \end{theorem}
 For $D= \{ h= 0\}$, we deduce the following natural commutative diagram,
 \begin{eqnarray}
 \xymatrix{(\Omega_{\mathcal{A}}^{\bullet}(log h),\epsilon^{*}_{\bullet}) \ar[r]^. \ar[dr]_{}& \textbf{R} Hom_{\mathcal{A}}(M^{log(h)}, \mathcal{A}) \ar[d]^ r\\
         &    DR( \widetilde{M}^{log(h)},\epsilon^{ }_{\bullet})}
 \end{eqnarray}
 where \cite{FJCJN} define $r$ by $r : \mathcal{A} \longrightarrow \widetilde{M}^{log(h)} $ such that $a\mapsto r(a) = [ah]$ and $[.]$ represent the equivalence class of the ideal $A_{p}\widetilde{Der}_{\mathcal{A}}(logh).$
 An explicit computation of the logarithmic de Rham cohomology is given below:
 \begin{theorem}\cite{KK} Let $\mathbb{F}$ be a field of characteristic zero.
 The logarithmic de Rham cohomology groups associated with $\{x,y\}=xy$ on $\mathcal{A}=\mathbb{F}[x,y]$ along $xy\mathcal{A}$ are
  $H^0_{_{DR-log}}\simeq \mathbb{F}$, $H^1_{_{DR-log}} \simeq \mathbb{F}\times\mathbb{F}$, $H^2_{_{DR-log}}\simeq\mathbb{F}$ and $H^k_{_{DR-log}}\simeq 0$, $k>2$.
 \end{theorem}
The difficulty encountered by \cite{KK} in determining the first classical Poisson cohomology $H^1_{Pois}$ is overcome by the following lemma.
\begin{lemma} \cite{KK} \label{lemmaKK}
Let $K,B, C$ be submodules of $\mathcal{A}$. If $B\subseteq K$ and $C\subseteq K$, then
\[
K\cap (B\oplus C)=(K\cap B)\oplus (K\cap C).
\]
\end{lemma}

\section{ Saito bases of some free inhomogeneous divisors}

\begin{definition}
 We define the inhomogeneous divisor by the linear combination of monomials involving
 the product of coordinates given by $\widetilde{D}=\{h= \overset{n}{\underset{i=1}{\sum}}
   \alpha_{i} x_{1}^{n_{1_{i}}}....x_{p}^{n_{p_{i}}}=0\}$ where $\alpha_{i} \neq 0$, $n_{j_{i}} >0$, $1\leq i \leq n$ and $1\leq j \leq p$.
 \end{definition}
In this section, we determine a Saito basis for inhomogeneous divisor $D$ in dimension $p\geq 2$ where $D= \{h= \overset{n}{\underset{i=1}{\sum}}
\alpha_{i}(\overset{p}{\underset{j=1}{\Pi}} x_{j})^{n_{i}} =0;
\alpha_{i}\in \mathbb{C}^{*}, n_{i}\neq n_{j}, \forall i\neq j\} \subset \widetilde{D}$.
 \begin{proposition} Let the inhomogeneous polynomial $h= \overset{n}{\underset{i=1}{\sum}}
\alpha_{i}(\overset{p}{\underset{j=1}{\Pi}} x_{j})^{n_{i}}$ where $ \alpha_{i}\in \mathbb{C}^{*}$ and $ n_{i}\neq n_{j}$ $\forall$ $i\neq j$.
 The divisor $D= \{h=0\}$ is not quasi-homogeneous.
 \end{proposition}

\begin{proof}
Suppose that $h$ is quasi-homogeneous of degree $W\in \mathbb{N}^{*}$ and denote the weight of $x_j$ by $\omega_{x_{j}}\in \mathbb{N}^{*}$. According to \textbf{Definition} $\ref{Def3}$, there exists a system consisting of p-tuples noted $supp(h) = \{(n_{1},...,n_{1}),..., (n_{p},...,n_{p})\}$ such that equations $n_{1}(\omega_{x_{1}}+...+\omega_{x_{p}})=...=n_{p}(\omega_{x_{1}}+...+\omega_{x_{p}})=W$ hold. It follows that $n_{1}=...=n_{p}$.
This is impossible because $n_{i}\neq n_{j}$, for all $i\neq j$. Thus, $h$ is not quasi-homogeneous, hence $D= \{h=0\}$ is not quasi-homogeneous. This completes the proof.
\end{proof}
 \begin{proposition}
 The divisor class $D=\{ h= \overset{n}{\underset{i=1}{\sum}}
   \alpha_{i} x_{1}^{n_{i}}...x_{p}^{n_{i}}=0; \alpha_{i} \neq 0, n_{i} >0 \} \subset \tilde{D}$ is not reduced if $inf(n_{i})_{i=1}^{n}$ is greater than or equal to $2$.
 \end{proposition}
 \begin{proof}
 If we take $n_{i_{0}}=inf(n_{i})$ with $1\leq i \leq n$, the monomial $x_{1}^{n_{i_{0}}}....x_{p}^{n_{i_{0}}}$ belongs to $(h)$ since $h= x_{1}^{n_{i_{0}}}...x_{p}^{n_{i_{0}}}\left(\alpha_{0} + \overset{n}{\underset{i=1}{\sum}}
   \alpha_{i} x_{1}^{{n_{i}}-n_{i_{0}}}...x_{p}^{n_{i}-n_{i_{0}}} \right)$. Hence each coordinate $x_{i}$'s in $x_{1}^{n_{i_{0}}}....x_{p}^{n_{i_{0}}}$ has multiplicity at least $2$. Therefore, $D$ is not reduced.
 \end{proof}
\begin{proposition} \label{proposition8}
The divisor $D=\{h=0\}$ where $h=\overset{n}{\underset{i=1}{\sum}} \alpha_{i} x^{i} y^{i} =th_{t}$, with $t=xy$, $\alpha_1\alpha_n\neq 0$. Assume that $\gcd(h_t,h_t')=1$ in $\mathbb{C}[t]$. Then, $h$ is a nonzero reduced polynomial. Consequently, $D=\{h=0\}$ is a nonzero reduced divisor in the affine plane $\mathbb{A}_\mathbb{C}^2$.
\end{proposition}
 \begin{proof}
 Since $\alpha_1\neq 0$ and $\alpha_n\neq 0$, hence $h_t$ is a nonzero polynomial of degree $n-1$. Consequently, $h=th_t\neq 0$. Further, $h_t(0)=h_0 = \alpha_1\neq 0$, thus $\gcd(t,h_t)=1$. In other words, $\gcd(h_t,h_t')=1$ means that $h_t$ has no multiple roots. Hence on $\mathbb{C}$ we have $h_t = \alpha_n \overset{n-1}{\underset{k=1}{\Pi}} (t-\lambda_k)$ where $\lambda_i\neq\lambda_j$. Furthermore, $h_0\neq 0$ implies that $\lambda_k\neq 0$ for any $k$.
 Thus, $h_t = \alpha_n \overset{n-1}{\underset{k=1}{\Pi}} (t-\lambda_k)$ where the polynomials $x$, $y$, and $xy-\lambda_k$ are pairwise non-associated irreducible factors in $\mathcal{A}$.
 Indeed, $\lambda_k\neq 0$ and $xy - \lambda_k$ is irreducible. Each factor therefore appears with multiplicity equal to 1. Then, $h$ is nonzero and reduced. Then, divisor $D$ is reduced.
  \end{proof}

It follows from the preceding proposition that the statements of \textbf{Proposition} \ref{proposition4} and \textbf{Theorem} \ref{Theorem4} hold for this specific divisor.
According to \cite{KS}, locally $hDer_{\mathcal{A}} \subset Der_{\mathcal{A}}(logD) \subset
Der_{\mathcal{A}}$ and $\Omega_{\mathcal{A}}^{1} \subset \Omega_{\mathcal{A}}^{1}(logD) \subset \dfrac{1}{h}\Omega_{\mathcal{A}}^{1}$. This allows us to define below the generating vector fields of $Der_{\mathcal{A}}(logD)$ and consequently those of $\Omega_{\mathcal{A}}^{1}(logD)$.
 \begin{proposition}
 Let $E_{p}= \overset{p}{\underset{i=1}{\sum}} x_{i}\partial _{x_{i}}$ be the local Euler vector field on $X=\mathbb{C}^{p}$.
 The vector field $\delta_{n} =
\overset{n}{\underset{i=1}{\sum}}
\alpha_{i}(\overset{p}{\underset{j=1,}{\Pi}} x_{j})^{n_{i}-1} E_{p}$
is logarithmic along $D = \{h=0\}$ where $ \overset{n}{\underset{i=1}{\sum}}
\alpha_{i}(\overset{p}{\underset{j=1}{\Pi}} x_{j})^{n_{i}} =0$; $\alpha_{i}\neq 0$, $n_{i}>0$, $n_{i}\neq n_{j}$.
 \end{proposition}
 \begin{proof}
Let $x$ an element of $D\subset X$ and $h_{x}$ the definition
hyperplane of $D$.
 Therefore, we have $\delta_{n} (h_{x} )=
\overset{n}{\underset{i=1}{\sum}}
\alpha_{i}n_{i}(\overset{p}{\underset{j=1}{\Pi}} x_{j})^{n_{i}-1}
\left(\overset{n}{\underset{i=1}{\sum}}
\alpha_{i}(\overset{p}{\underset{j=1}{\Pi}} x_{j})^{n_{i}} \right)
= \overset{n}{\underset{i=1}{\sum}}
\alpha_{i}n_{i}(\overset{p}{\underset{j=1}{\Pi}}
x_{j})^{n_{i}-1}h_{x}$. It follows that $\delta_{n}
(h_{x})$ belongs to the ideal $( h_{x})$. And then $\delta_{n}$ is logarithmic along
$D$.
 \end{proof}
 \begin{proposition} $\label{Prop7}$
 Consider $\overline{q}=(q_{1},...,q_{n}) \in \mathbb{N}^{n}$ such that
  $\overset{n}{\underset{i=1}{\sum}}q_{i}=n$
  and the vector field \\
  $\delta_{n}^{\overline{q}}=
  \overset{q_{1}}{\underset{i=1}{\sum}}
  \alpha_{i}(\overset{p}{\underset{j=1,}{\Pi}} x_{j})^{n_{i}-1}
  x_{1}\partial _{x_{1}} +
  \overset{q_{2}}{\underset{i=q_{1}+1}{\sum}}
  \alpha_{i}(\overset{p}{\underset{j=1,}{\Pi}} x_{j})^{n_{i}-1}
  x_{2}\partial _{x_{2}} +... +
  \overset{q_{n}}{\underset{i=q_{n-1}}{\sum}}
  \alpha_{i}(\overset{p}{\underset{j=1,}{\Pi}} x_{j})^{n_{i}-1}
  x_{n}\partial _{x_{n}}$.
 Then
 $\delta_{n}^{\overline{q}}$ is logarithmic along $D =\{h= \overset{n}{\underset{i=1}{\sum}}
  \alpha_{i}(\overset{p}{\underset{j=1}{\Pi}} x_{j})^{n_{i}} =0; \alpha_{i}\neq 0, n_{i}>0, n_{i}\neq n_{j}\}$.
 \end{proposition}
 \begin{proof}
 Let $x$ an element of $D$ and $h_{x}$ the definition hyperplane of
 $D$.
 We get after computing $\delta_{n}^{\overline{q}} (h_{x}) $ that $\delta_{n}^{\overline{q}} (h_{x}) = \delta (h_{x}) \in ( h_{x})$.
 It follows from the preceding discussion that the
 vector field $\delta_{n}^{\overline{q}}$ is logarithmic along
 $D$.
 \end{proof}
 \begin{proposition}
 Let the divisor $D =\{h= \overset{n}{\underset{i=1}{\sum}}
   \alpha_{i}(\overset{p}{\underset{j=1}{\Pi}} x_{j})^{n_{i}} =0; \alpha_{i}\neq 0, n_{i}>0, n_{i}\neq n_{j}\}$ and the submodule $\mathcal{M} = \{ \delta^{0} \in Der_{\mathcal{A}}(logD)~ / ~~\delta^{0}(h) = 0\}$ of $Der_{\mathcal{A}}(logD)$. There is an isomorphism of $\mathcal{A}$-modules from $Der_{\mathcal{A}}(logD)$ to $\mathcal{O}_{\mathcal{A}}(\delta_{n} + \delta_{n}^{\overline{q}}) \oplus \mathcal{M}$ where $\mathcal{M}$ is given by the following:\\
    $\mathcal{M} = \{ \overset{ }{ \underset{0 < i; j \leq n}{\sum}}\alpha_{ij}\delta^{0}_{ij} \in Der_{\mathcal{A}}(logD), \alpha_{ij}\in \mathcal{A} / \delta^{0}_{ij} = (-1)^{i}x_{i}\partial_{x_{i}} + (-1)^{i+1}x_{j}\partial _{x_{j}} \}$.
 \end{proposition}
 \begin{proof}
 $\delta_{n}$, $\delta_{n}^{\overline{q}}$ and $\delta^{0}$ are logarithmic along $D$. Hence $\mathcal{O}_{\mathcal{A}}(\delta_{n} + \delta_{n}^{\overline{q}}) \oplus \mathcal{M} \subset Der_{\mathcal{A}}(logD)$. Now, let $\delta \in Der_{\mathcal{A}}(logD)$. We have $(\delta_{n}+\delta_{n}^{\overline{q}})h=u.h$, where $u$ is unit. In other words, $u^{-1}(\delta_{n}+\delta_{n}^{\overline{q}})h=h$. Since $\delta \in Der_{\mathcal{A}}(logD)$, thus there exists $a\in \mathcal{A}$ such that $\delta(h)=a.h$, which implies $\delta(h)=a.u^{-1}(\delta_{n}+\delta_{n}^{\overline{q}})h$, it follows that $\delta - a.u^{-1}(\delta_{n}+\delta_{n}^{\overline{q}}) \in \mathcal{M}$. Therefore, we decompose $\delta\in Der_{\mathcal{A}}(logD)$ as $\delta = u^{-1}(\delta_{n}+\delta_{n}^{\overline{q}})+ \left(\delta - u^{-1}(\delta_{n}+\delta_{n}^{\overline{q}})\right) \in \mathcal{O}_{\mathcal{A}}(\delta_{n} + \delta_{n}^{\overline{q}})\oplus \mathcal{M}$. Thus, we conclude that $Der_{\mathcal{A}}(logD)\simeq \mathcal{O}_{\mathcal{A}}(\delta_{n} + \delta_{n}^{\overline{q}}) \oplus \mathcal{M}$
 \end{proof}
 For example, let $h = \alpha_{1}x_{1}^{n_{1}} x_{2}^{n_{1}}x_{3}^{n_{1}} + \alpha_{2}x_{1}^{n_{2}} x_{2}^{n_{2}}x_{3}^{n_{2}} + \alpha_{3} x_{1}^{n_{3}} x_{2}^{n_{3}}x_{3}^{n_{3}}$ on $\mathbb{C}[x_{1},x_{2},x_{3}]$. The vector fields given by $\delta^{0}_{12} = x_{1}\partial _{x_{1}} - x_{2}\partial_{ x_{2}}$, $\delta^{0}_{13} = x_{1}\partial _{x_{1}} - x_{3}\partial _{x_{3}}$, $\delta^{0}_{23} = x_{2}\partial _{x_{2}} - x_{3}\partial_{ x_{3}}$ and $\delta^{0} = \tilde{ \alpha}_{12} \delta^{0}_{12} + \tilde{ \alpha}_{13} \delta^{0}_{13} + \tilde{ \alpha}_{23} \delta^{0}_{23}$
 for all $ \tilde{ \alpha}_{ij}= \alpha_{ij}- \alpha_{ji} \in \mathcal{A}$ belong to $\mathcal{M}$. We have $\delta^{0}_{ij}(h)= \delta^{0}(h)=0$.
 A system of $p$ distinct vector fields $\delta_{n}^{\overline{q}}$ constitutes
 what we call the K. Saito basis of $Der_{\mathcal{A}}(logD)$.
 \begin{corollary} Let $\delta^{0} = \overset{}{\underset{0 < i; j \leq p}{\sum}}\alpha_{ij}\delta^{0}_{ij}$, $\alpha_{ij}\in \mathcal{A}$,
\item[1.] The vector field $\delta^{0}$ is logarithmic along $D$ since all the vectors fields $\delta^{0}_{ij}$ are,
 \item[2.] The system $\mathcal{B} = \langle \delta^{1} ,..., \delta^{p}\rangle$ is a free Saito basis of $Der_{\mathcal{A}}(logD)$ if and only if its dual $\mathcal{B}^{*}= \langle \omega_{1},...,\omega_{p} \rangle$ is also a free Saito basis of $\Omega_{\mathcal{A}}^{1}(logD)$.
\end{corollary}
In the following we will implement the above method to determination of the Saito bases of $Der_{\mathcal{A}}(logD)$ and $\Omega_{\mathcal{A}}^{1}(logD)$ for all $h \in \mathcal{A} = \mathbb{C}[x_{1},x_{2},x_{3}]$.

\begin{proposition}
Let $D=\{h = \alpha_{1}x_{1}^{n_{1}} x_{2}^{n_{1}}x_{3}^{n_{1}} +
\alpha_{2}x_{1}^{n_{2}} x_{2}^{n_{2}}x_{3}^{n_{2}} + \alpha_{3}
x_{1}^{n_{3}} x_{2}^{n_{3}}x_{3}^{n_{3}} = 0\}$ and the vector fields
$\delta^{3} = \delta_{3}^{1,1,1} = \alpha_{1}x_{1}^{n_{1}-1} x_{2}^{n_{1}}x_{3}^{n_{1}}\partial_{x_{1}} +
 \alpha_{2}x_{1}^{n_{2}} x_{2}^{n_{2}-1}x_{3}^{n_{2}} \partial_{x_{2}}+ \alpha_{3}
 x_{1}^{n_{3}} x_{2}^{n_{3}}x_{3}^{n_{3}-1} \partial_{x_{3}}$, $\delta^{2}=x_{2}\partial_{x_{2}} - x_{3}\partial_{x_{3}}$ and
 $\delta^{1}=x_{1}\partial_{x_{1}} - x_{2}\partial_{x_{2}}$. Then the system $\mathcal{B} = \langle \delta^{1}, \delta^{2}, \delta^{3}\rangle$ constitutes what we call a free Saito basis of $Der_{\mathcal{A}}(logD)$.
\end{proposition}
    \begin{proof}
 First, we verify that $\mathcal{B} = \langle \delta^{1},\delta^{2},\delta^{3} \rangle$
 is the Saito basis of $Der_{\mathcal{A}}(logD)$ since, for all $i$ we have $\delta^{i}(h)\in (h)$ with $\delta^{1}\wedge \delta^{2}\wedge \delta^{3} = h\partial _{x_{1}}\wedge \partial _{x_{2}} \wedge \partial _{x_{3}}$.
And second, it is a free basis because
$det(\textbf{A})=h$ is a unit multiple of $h$. $\textbf{A}$ is called the Saito matrix of divisor $D$ given by
  $ \textbf{A} = \left(
     \begin{array}{ccc}
    x_{1} & 0 & \overline{a}_{1}\\
    -x_{2} & x_{2} & \overline{a}_{2}\\
   0 & -x_{3} & \overline{a}_{3}
     \end{array}
     \right)$ where $a_{i}= \alpha_{i}x_{1}^{n_{i}-1} x_{2}^{n_{i}-1}x_{3}^{n_{i}-1}$, $\overline{a}_{i}=a_{i}x_{i}$.
     The freeness of $Der_{\mathcal{A}}(logD)$ comes from the fact that $det (\textbf{A}) = h$.
      This completes the proof.
    \end{proof}

\begin{proposition} The dual basis to $\langle \delta^{1}, \delta^{2}, \delta^{3}\rangle$
is $ \langle \omega_{1}, \omega_{2}, \omega_{3}\rangle$ given by the following

$\omega_{1} = \frac{x_{1}^{n_{2}-1} x_{2}^{n_{2}}x_{3}^{n_{2}} + x_{1}^{n_{3}-1} x_{2}^{n_{3}}x_{3}^{n_{3}}}{h}dx_{1} - \frac{x_{1}^{n_{1}} x_{2}^{n_{1}-1}x_{3}^{n_{1}}}{h}dx_{2}  - \frac{x_{1}^{n_{1}} x_{2}^{n_{1}}x_{3}^{n_{1}-1}}{h}dx_{3}$\\

$\omega_{2} = \frac{x_{1}^{n_{3}-1} x_{2}^{n_{3}}x_{3}^{n_{3}}}{h}dx_{1} - \frac{x_{1}^{n_{3}} x_{2}^{n_{3}-1}x_{3}^{n_{3}}}{h}dx_{2}  - \frac{x_{1}^{n_{1}} x_{2}^{n_{1}}x_{3}^{n_{1}-1} + x_{1}^{n_{2}} x_{2}^{n_{2}}x_{3}^{n_{2}-1}}{h}dx_{3}$\\

$ \omega_{3} = \dfrac{x_{2}x_{3}}{h}dx_{1} +
\dfrac{x_{1}x_{3}}{h}dx_{2} + \dfrac{x_{1}x_{2}}{h}dx_{3}.$
\end{proposition}

\begin{proof}
Let $\omega_{1} = \omega_{1}^{1}dx_{1} + \omega_{1}^{2}dx_{2} +
\omega_{1}^{3}dx_{3}$, $\omega_{2} = \omega_{2}^{1}dx_{1} +
\omega_{2}^{2}dx_{2} + \omega_{2}^{3}dx_{3} $ and $\omega_{3} =
\omega_{3}^{1}dx_{1} + \omega_{3}^{2}dx_{2} +
\omega_{3}^{3}dx_{3}$ elements of $\Omega_{\mathcal{A}}^{1}(logD)$.
According to \textbf{Proposition} \ref{P3},
 the condition $\delta^{1}\lrcorner \omega_{j }$ give us
 $\left(
 \begin{array}{ccc}
 \omega_{1}^{1} & - \omega_{1}^{2} & 0\\
\omega_{2}^{1} & - \omega_{2}^{2} & 0\\
\omega_{3}^{1} & - \omega_{3}^{2} & 0\\
  \end{array}
  \right) \left(
    \begin{array}{cc}
   x_{1}\\
   x_{2}\\
   x_{3}
    \end{array}
    \right) = \left(
       \begin{array}{cc}
      1\\
      0\\
      0
       \end{array}
       \right).$
      Similarly,
 $\delta^{2}\lrcorner \omega_{j} $ implies
 $\left(
\begin{array}{ccc}
 0 & \omega_{1}^{2} & -\omega_{1}^{3} \\
 0 & \omega_{2}^{2} & - \omega_{2}^{3}\\
 0 & \omega_{3}^{2} & - \omega_{3}^{3}\\
 \end{array}
  \right) \left(
 \begin{array}{cc}
 x_{1}\\
 x_{2}\\
 x_{3}
 \end{array}
 \right) = \left(
 \begin{array}{cc}
 0\\
 1\\
 0
\end{array}
 \right)$
and
 $\delta^{3}\lrcorner \omega_{j} $ means that
 $\left(
\begin{array}{ccc}
 \tilde{a}_{11} &\tilde{a}_{12} & \tilde{a}_{13} \\
 \tilde{a}_{21} & \tilde{a}_{22} & \tilde{a}_{23} \\
 \tilde{a}_{31}  & \tilde{a}_{32} & \tilde{a} _{33} \\
 \end{array}
  \right) \left(
 \begin{array}{cc}
 x_{1}\\
 x_{2}\\
 x_{3}
 \end{array}
 \right) = \left(
 \begin{array}{cc}
 0\\
 0\\
 1
\end{array}
 \right)$ where $\tilde{a}_{ij} = a_{j} \omega_{i}^{j} $.
Therefore we get any $\omega_{i}^{j}$ given in the previous
proposition. According to \textbf{Theorem} \ref{T1} we verify
the Saito criterion $\omega_{1} \wedge\omega _{2}\wedge\omega _{3}=
\dfrac{1}{h} dx_{1}\wedge dx_{2}\wedge dx_{3}$.
Furthermore $\delta^{i}\lrcorner
\omega_{j}= \delta_{ij}$. We conclude that $\mathcal{B}^{*} =
\langle \omega_{1}, \omega_{2}, \omega_{3}\rangle$ is a free Saito
basis of $\Omega_{\mathcal{A}}^{1}(logD)$, dual of
$Der_{\mathcal{A}}(logD)$.
\end{proof}
The higher-dimensional construction is used to motivate the Saito-theoretic framework, all cohomological computations in this paper are carried out in dimension two.
We focus on the Poisson algebra $(\mathcal{A}=\mathbb{C}[x,y], \{.,.\}_{h})$ where
 $h=\overset{n}{\underset{i=1}{\sum}}
 \alpha_{i} x^{i} y^{i}$ with $h=xyh_{xy}$, $h_{t}=\overset{n}{\underset{i=1}{\sum}}
 \alpha_{i} (t)^{i-1}$, $\gcd(h_t,h_t')=1$ in $\mathbb{C}[t]$ and $\alpha_1\alpha_n\neq 0$.
 $\mathcal{I}=h\mathcal{A}$ is an ideal of $\mathcal{A}$, $Der_\mathcal{A}(\log \mathcal{I})$ is the module of logarithmic vector fields along $\mathcal{I}$ and $\Omega_{\mathcal{A}}^{1}(log \mathcal{I})$ is that of 1-forms logarithmic along $\mathcal{I}$.

\section{Logarithmic Koszul Bracket Associated with $\pi=h\partial_x\wedge\partial_y$}

\begin{lemma} $\label{31}$
 Let $E_{2}=x\partial_x + y\partial_y$, $R = y\partial_y - x\partial_x$. Consider
 $\mathcal{B}=( \delta^{1} = h_{xy}E_{2},~ \delta^{2} = R)$
and
 $\mathcal{B}^*=(\omega_{1} = \dfrac{y}{2h}dx +
\dfrac{x}{2h}dy,~ \omega_{2}= \dfrac{dy}{2y} - \dfrac{dx}{2x})$.
Then,
 $\mathcal{B}$ is a Saito basis of $Der_{\mathcal{A}}(log
\mathcal{I})$ and $\mathcal{B}^*$ is a dual Saito basis of
$\Omega_{\mathcal{A}}^{1}(log \mathcal{I})$,
\end{lemma}

\begin{proof}
 First, we have $\delta^{1}, \delta^{2} \in \mathcal{B}$, $\delta^{1} (h)=(2 \overset{n}{\underset{i=1}{\sum}}
 i \alpha_{i} x^{i-1} y^{i-1}) h \in (h)$, $\delta^{2}(h) = 0 \in (h)$ and second, $det(\delta^{1},\delta^{2}) = h$ is a unit multiple of $h$, thus $\mathcal{B}$ is a free Saito basis of $Der_{\mathcal{A}}(log
\mathcal{I})$. Otherwise, $\omega_{1},~ \omega_{2} \in
\mathcal{B}^*$ with $\omega_{1} \wedge \omega_{2}=
\dfrac{(\frac{1}{2})}{h} dx \wedge dy = \dfrac{unit}{h} dx \wedge
dy$.
 Furthermore $\mathcal{B}$ and $\mathcal{B}^*$ are duals, since $\delta^{i} \lrcorner \omega_{j}=\delta_{ij}$ for
 all $i,j=1;2$ where $\delta_{ij}$ is the Kronecker's symbol.
This completes the proof.
\end{proof}
From the preceding discussion lemma, note that $Der_{\mathcal{A}}(log \mathcal{I}) = \mathcal{A}\delta^1 \oplus \mathcal{A} \delta^2$ and $\Omega_{\mathcal{A}}^{1}(log \mathcal{I}) = \mathcal{A}\omega_{1} \oplus \mathcal{A} \omega_{2}$.
 Let $H$ to be the standard Hamiltonian map and $\mathcal{A}_h=\mathcal{A}[h^{-1}]$ the localization of $h$.
\begin{definition}
 The Hamiltonian localization map $H_h$ induced by the Hamiltonian map $H$ is $H_h: \Omega_\mathcal{A}\otimes_{_{\mathcal{A}}}\mathcal{A}_h \longrightarrow Der_\mathcal{A}\otimes_{_{\mathcal{A}}}\mathcal{A}_h$ such that $H_h(\omega \otimes \frac{a}{h}) = H(\omega) \otimes \frac{a}{h}$ for any $\omega\in \Omega_\mathcal{A}$ and $a\in \mathcal{A}$.
\end{definition}
 We have the natural identification $\Omega_\mathcal{A}\otimes_{_{\mathcal{A}}}\mathcal{A}_h\simeq \Omega_{\mathcal{A}_{h}}$
and $Der_\mathcal{A}\otimes_{_{\mathcal{A}}}\mathcal{A}_h\simeq Der_{\mathcal{A}_{h}}$ and the Hamiltonian localization map becomes $H_h: \Omega_{\mathcal{A}_{h}}\longrightarrow Der_{\mathcal{A}_{h}}$. Equivalently, if $\eta =\frac{\omega}{h^{k}}\in \Omega_{\mathcal{A}_{h}}$ then $H_h(\eta)=\frac{1}{h^{k}}H(\omega)$. By definition, $H_h$ is a localized extension of $H$.

\begin{lemma} \label{lem2}
Let $H_h$ be the localized Hamiltonian map of $H$. Then we have, $$H_h(\Omega_{\mathcal{A}}^1(log \mathcal{I})) \subseteq Der_{\mathcal{A}}(log \mathcal{I}).$$
\end{lemma}

\begin{proof} Remember that $\Omega_\mathcal{A}^1(log \mathcal{I})=\mathcal{A}\omega_1\oplus\mathcal{A}\omega_2$ and $Der_\mathcal{A}(log \mathcal{I})=\mathcal{A}\delta^1\oplus\mathcal{A}\delta^2$.
Consider a 1-form $\omega=a\omega_1 + b \omega_2\in \Omega_\mathcal{A}^1(log \mathcal{I})$ for all $a,b\in \mathcal{A}$. Thus $H_h(\omega) = H_h(a\omega_1) + H_h(b\omega_2)$. Furthermore, we have $H_h(a\omega_1)=H_h(\frac{ay}{2h}dx + \frac{ax}{2h}dy) = \frac{ay}{2h}H(dx) + \frac{ax}{2h}H(dy)$. By the same method, one verifies that
$H_h(b\omega_2)= \frac{b}{2y}H(dy) - \frac{b}{2x}H(dx)$.
Since $H(dx)=h\partial_y$ and $H(dy)= -h\partial_x$, it follows that $H_h(\omega_1) = \frac{1}{2}\delta^2$ and $H_h(\omega_2) = -\frac{1}{2}\delta^1$. Therefore we obtain $H_h(\omega) = \frac{1}{2}(a\delta^2 - b \delta^1)\in \Omega_{\mathcal{A}}^1(log \mathcal{I})$. Hence $H_h(\Omega_{\mathcal{A}}^1(log \mathcal{I})) \subseteq Der_{\mathcal{A}}(log \mathcal{I})$. This completes the proof.
\end{proof}

\begin{proposition}\label{proposition17}
 There exists a well-defined $\mathcal{A}$-linear map $\pi_{\log}^{\sharp} :=\widetilde{H}$ given by the following
 $\pi_{\log}^{\sharp} :=\widetilde{H}:\quad \Omega_{\mathcal{A}}^1(log \mathcal{I}) \longrightarrow Der_{\mathcal{A}}(log \mathcal{I})$ which extends $H$.
\end{proposition}

\begin{proof}
The existence of $\pi_{\log}^{\sharp}$ comes from \textbf{Lemma} \ref{lem2} such that $\pi_{\log}^{\sharp}(\alpha)=H_h(\alpha)$ for  $\alpha\in \Omega_{\mathcal{A}}$. We will now prove that $\pi_{\log}^{\sharp}$ is well-defined. Let us consider
$
\iota_\Omega:\Omega_\mathcal{A}\hookrightarrow\Omega_\mathcal{A}^1(\log \mathcal{I})$
and
$
\iota_{Der}: Der_\mathcal{A}(log \mathcal{I})\hookrightarrow Der_\mathcal{A}$
be the natural inclusion maps.
Since $H_h$ is the localization of $H$, we have $H_h(\alpha)=H(\alpha)$.
Hence $\pi_{\log}^{\sharp}(\alpha)=H(\alpha)$, for all $\alpha\in\Omega_\mathcal{A}$.
Equivalently, $\pi_{\log}^{\sharp}|_{\Omega_\mathcal{A}}=H$.
Furthermore, $\widetilde H\bigl(\iota_\Omega(\alpha)\bigr)=\iota_{Der}\bigl(H(\alpha)\bigr)$.
Therefore, $\pi_{\log}^{\sharp}:\Omega_\mathcal{A}^1(log D)\longrightarrow Der_\mathcal{A}(log D)$ is a well-defined logarithmic Hamiltonian map which extends $ H:\Omega_\mathcal{A}\longrightarrow Der_\mathcal{A}$.
\end{proof}

\begin{corollary}
 For all $\omega =
a\omega_{1} + b \omega_{2}\in \Omega_{\mathcal{A}}^{1}(log\mathcal{I})$ and $a,b \in \mathcal{A}=\mathbb{C}[x,y]$, the explicit expression of $\pi_{\log}^{\sharp}$ is
    $\pi_{\log}^{\sharp}(\omega)= \frac{1}{2}(a \delta^{2}- b \delta^{1}) \in Der_{\mathcal{A} }(Log \mathcal{I})$.
\end{corollary}

\begin{definition} For $\delta\in Der_{\mathcal{A}}(log\mathcal{I})$,
the logarithmic interior product is the $\mathcal{A}$-linear map
$\iota_{\delta}: \Omega_{\mathcal{A}}^{p}(log\mathcal{I}) \longrightarrow \Omega_{\mathcal{A}}^{p-1}(log\mathcal{I})$ such that $(\iota_{\delta}\omega) (\delta^1,...,\delta^{p-1})= \omega(\delta, \delta^1,...,\delta^{p-1})$.
\end{definition}
 $\iota_{\delta}(f) =0$ and $\iota_{\delta}(\alpha) = \alpha (\delta)$ with $\alpha \in \Omega_{\mathcal{A} }^{1}(log \mathcal{I})$, $f\in \mathcal{A}$. The graded Leibniz rule is
 \begin{equation}
\iota_{\delta}(\alpha\wedge \beta) = \iota_{\delta}(\alpha) \wedge \beta + (-1)^p\alpha \wedge \iota_{\delta}(\beta).
 \end{equation}
 In particular $\iota_{\delta^1}(\omega_1\wedge \omega_2)=\omega_2$ and $\iota_{\delta^2}(\omega_1\wedge \omega_2)=-\omega_1$. The Lie derivative $\mathcal{L}_{\delta}$ is defined by
\begin{equation} \label{relation12}
  \mathcal{L}_{\delta}= \iota_{\delta} \circ \widetilde{d} + \widetilde{d}\circ \iota_{\delta}.
\end{equation}
 It follows from this relation that $\mathcal{L}_{\delta}\circ \widetilde{d} = \widetilde{d}\circ \mathcal{L}_{\delta}$.
Since $\Omega_{\mathcal{A}} \subset \Omega_{\mathcal{A}}^1(log \mathcal{I})$, the Hamiltonian map $\sharp :=H$ given in relation (\ref{r2} or \ref{r5}) is extended according to \textbf{Proposition} \ref{proposition17} by $\widetilde{H}:=\pi^{\sharp}_{\log}$. Then, the Koszul bracket given in \textbf{Definition} \ref{d3} becomes as follows.

\begin{definition} \label{d14} Let $\pi = h\partial_{x}\wedge \partial_{y}$.
 The logarithmic Koszul bracket of two logarithmic 1-forms is defined on $\Omega_\mathcal{A}^1(log \mathcal{I})$ by $[.,.]_{\Omega_\mathcal{A}^1(log \mathcal{I})} : \Omega_\mathcal{A}^1 (log \mathcal{I})\times \Omega_\mathcal{A}^1(log \mathcal{I}) \longrightarrow \Omega_\mathcal{A}^1(log \mathcal{I})$ such that
 $[\omega_{i}, \omega_{j}]_{\Omega_\mathcal{A}^1(log \mathcal{I})}=\mathcal{L}_{\pi^{\sharp}_{\log}(\omega_{i})}(\omega_{j})-\mathcal{L}_{\pi^{\sharp}_{\log}(\omega_{j})}(\omega_{i})-\widetilde{d}\pi(\omega_{i},\omega_{j})$, for all $\omega_{i}, \omega_{j} \in \Omega_{\mathcal{A}}^1 (log \mathcal{I})$.
 \end{definition}

\begin{lemma} \label{lemma4}
Let $[.,.]_{\Omega_{\mathcal{A}}^{1}(log\mathcal{I})} $ the induced Koszul bracket on $\Omega_{\mathcal{A}}^{1}(log\mathcal{I}) = \mathcal{A} \omega_{1} \oplus \mathcal{A} \omega_{2}$ by the Poisson bivector
$\pi = h\partial_x\wedge\partial_y$. Then, $[\omega_{1},\omega_{2}]_{\Omega_{\mathcal{A}}^{1}(log\mathcal{I})}=0$.
\end{lemma}

 \begin{proof}

  For all 1-forms $\eta_{1}= \eta_{1}^{1} \omega_{1} + \eta_{1}^{2} \omega_{2},\eta_{2} = \eta_{2}^{1} \omega_{1} + \eta_{2}^{2} \omega_{2}$ we can also define the logarithmic Koszul bracket definition by
  $[.,.]_{\Omega_{\mathcal{A}}^{1}(log\mathcal{I})}: \Omega_{\mathcal{A}}^{1}(log \mathcal{I}) \times \Omega_{\mathcal{A}}^{1}(log \mathcal{I}) \longrightarrow \Omega_{\mathcal{A}}^{1}(log \mathcal{I}) $ such that $[\eta_{1},\eta_{2}]_{\Omega_{\mathcal{A}}^{1}(log\mathcal{I})} = \psi^{1} \omega _{1} + \psi^{2} \omega _{2} \in \Omega_{\mathcal{A}}^{1}(log \mathcal{I})$ where $\langle [\eta_{1}, \eta_{2}]_{\Omega_{\mathcal{A}}^{1}(log \mathcal{I})}, \omega_{i}^{*} \rangle = \psi^{i} \in Der_{\mathcal{A}}(log \mathcal{I})$ with $\omega_{i}^{*}=\delta^{i}$.
  Further, $[\omega_{1},\omega_{2}]_{\Omega_{\mathcal{A}}^{1}(log\mathcal{I})} = \phi^{1} \omega _{1} + \phi^{2} \omega _{2}$; $\phi^{i} \in \mathcal{A}$.
   According to \textbf{Definition} \ref{d14} we get,
  \begin{equation}
  \left( \mathcal{L}_{\pi^{\sharp}_{\log}(\omega_{1})}(\omega_{2})-\mathcal{L}_{\pi^{\sharp}_{\log}(\omega_{2})}(\omega_{1})-\widetilde{d}\pi(\omega_{1},\omega_{2})\right)(\omega_{i}^{*}) = \phi^{i}
  \end{equation}
 After computing, we get with relation (\ref{relation12}) that $\mathcal{L}_{\pi^{\sharp}_{\log}(\omega_{1})}(\omega_{2})=\mathcal{L}_{\pi^{\sharp}_{\log}(\omega_{2})}(\omega_{1})=0$. Furthermore, $\pi (\omega_{1},\omega_{2}) = h(\dfrac{y}{2h}.\dfrac{1}{2y} - \dfrac{x}{2h}.(-\dfrac{1}{2x}))= \dfrac{1}{2}$ implies that
 $d\pi (\omega_{1}, \omega_{2})=0$. Therefore $\phi^{i} =0$ and it follows that $[\omega_{1},\omega_{2}]_{\Omega_{\mathcal{A}}^{1}(log\mathcal{I})} =0$.
 \end{proof}

 \begin{lemma} \label{lemma5}
 For all $\alpha,\beta\in\Omega_\mathcal{A}^1(\log \mathcal{I})$ and $a\in \mathcal{A}$, the Koszul bracket $[.,.]_{\Omega_{\mathcal{A}}^{1}(log\mathcal{I})}$ verifies
 $$[\alpha,a\beta]_{\Omega_{\mathcal{A}}^{1}(log\mathcal{I})}
 =
 a[\alpha,\beta]_{\Omega_\mathcal{A}^1(\log\mathcal{I})}
 +
 \pi^\sharp_{\log}(\alpha)(a)\,\beta.$$
 \end{lemma}

 \begin{proof}
 By definition,
 we have $[\alpha,\beta]_{\Omega_{\mathcal{A}}^{1}(log\mathcal{I})}
 =
 \mathcal L_{\pi^\sharp_{\log}(\alpha)}\beta
 -
 \mathcal L_{\pi^\sharp_{\log}(\beta)}\alpha
 -
 \widetilde{d}\bigl(\pi(\alpha,\beta)\bigr)$.
 Since $\pi^\sharp_{\log}$ and $\pi$ are $\mathcal{A}$-linears, we get
$ \pi^\sharp_{\log}(a\beta)=a\pi^\sharp_{\log}(\beta)$ and
$ \pi(\alpha,a\beta)=a\pi(\alpha,\beta)$.
 Let us consider
 $Z=\pi^\sharp_{\log}(\alpha)$ and
$ Y=\pi^\sharp_{\log}(\beta)$.
 Therefore
 $ [\alpha,a\beta]_{\Omega_{\mathcal{A}}^{1}(log\mathcal{I})}
 =
 \mathcal L_Z(a\beta)-\mathcal L_{aY}\alpha
 - \widetilde{d}\!\left(a\pi(\alpha,\beta)\right)$.
 According to relations given in (\ref{relation6}), (\ref{relation7}), $ \mathcal L_Z(a\beta)=Z(a)\beta+a\mathcal L_Z\beta$ and $\mathcal L_{aY}\alpha
 =a\mathcal L_Y\alpha+\alpha(Y)\widetilde{d}a.$
 Furthermore, we apply the derivative property of $\widetilde{d}$ to get
 $\widetilde{d}\!\left(a\pi(\alpha,\beta)\right)
 =
 a\,\widetilde{d}\bigl(\pi(\alpha,\beta)\bigr)
 +\pi(\alpha,\beta)\widetilde{d}a$. Thus
 we obtain $[\alpha,a\beta]_{\Omega_{\mathcal{A}}^{1}(log\mathcal{I})}
 ={}Z(a)\beta+a\mathcal L_Z\beta-a\mathcal L_Y\alpha
 -\alpha(Y)\widetilde{d}a
 -a\,\widetilde{d}\bigl(\pi(\alpha,\beta)\bigr)
 -\pi(\alpha,\beta)\widetilde{d}a.$
 Because
$\alpha(Y)
 =
 \alpha(\pi^\sharp_{\log}(\beta))
 =
 \pi(\beta,\alpha)
 =
 -\pi(\alpha,\beta)$,
 the two terms involving $\widetilde{d}a$ cancel. Hence
$[\alpha,a\beta]_{\Omega_{\mathcal{A}}^{1}(log\mathcal{I})}
 =
 Z(a)\beta
 +
 a[\alpha,\beta]_{\Omega_{\mathcal{A}}^{1}(log\mathcal{I})}$.
 Finally, $Z=\pi^\sharp_{\log}(\alpha)$, so it follows that
 \[
 [\alpha,a\beta]_{\Omega_{\mathcal{A}}^{1}(log\mathcal{I})}
 =
 a[\alpha,\beta]_{\Omega_{\mathcal{A}}^{1}(log\mathcal{I})}
 +
 \pi^\sharp_{\log}(\alpha)(a)\beta.
 \]
 This proves the lemma.
 \end{proof}

\begin{proposition} $\label{proposition19}$
For all $a,b\in \mathcal{A}$,
$[a\omega_1,b\omega_2]_{\Omega_{\mathcal{A}}^{1}(log\mathcal{I})}=\frac{a}{2}\delta^2(b)\,\omega_2+ \frac{b}{2}\,\delta^1(a)\,\omega_1$.
\end{proposition}

\begin{proof}
With \textbf{Lemma} \ref{lemma5}, the Leibniz rule for the logarithmic Koszul bracket is given by the relation
$[\alpha,a\beta]_{\Omega_{\mathcal{A}}^1(log \mathcal{I})}
=
a[\alpha,\beta]_{\Omega_{\mathcal{A}}^1(log \mathcal{I})}
+
\pi^\sharp_{\log}(\alpha)(a)\beta$.
Applying this rule twice gives the following
\begin{align*}
[a\omega_1,b\omega_2]_{\Omega_{\mathcal{A}}^1(log \mathcal{I})}
={}&
a[\omega_1,b\omega_2]_{\Omega_{\mathcal{A}}^1(log \mathcal{I})}
-
b\pi^\sharp_{\log}(\omega_2)(a)\omega_1
\\
={}&
ab[\omega_1,\omega_2]_{\Omega_{\mathcal{A}}^1(log \mathcal{I})}
+
a\pi^\sharp_{\log}(\omega_1)(b)\omega_2
-
b\pi^\sharp_{\log}(\omega_2)(a)\omega_1.
\end{align*}
We now compute the three terms above. Using $\pi=h\,\partial_x\wedge\partial_y$,
we obtain
$\pi^\sharp_{\log}(\omega_1)
= \frac12\delta^2$.
Indeed, $\omega_1=\frac{y}{2h}\,dx+\frac{x}{2h}\,dy$, and therefore $\iota_{\omega_1} \left(h\,\partial_x\wedge\partial_y\right)
= \frac12 \left(-x\partial_x+y\partial_y\right)$ with the stated contraction convention.
Similarly, $\pi^\sharp_{\log}(\omega_2) =-\frac12\delta^1$. Indeed, $\omega_2=\frac{dy}{2y} -\frac{dx}{2x}$, hence we have $\pi^\sharp_{\log}(\omega_2) = -\frac h2 \left(\frac1y\partial_x+\frac1x\partial_y \right) = -\frac12\delta^1$.
It remains to verify that $[\omega_1,\omega_2]_{\Omega_{\mathcal{A}}^1(log \mathcal{I})}=0$. This is proved in the \textbf{Lemma} \ref{lemma4}.
Consequently, $[a\omega_1,b\omega_2]_{\Omega_{\mathcal{A}}^1(log \mathcal{I})} = a\pi^\sharp_{\log}(\omega_1)(b)\omega_2- b\pi^\sharp_{\log}(\omega_2)(a)\omega_1$.
Using the relations $\pi^\sharp_{\log}(\omega_1)=\frac12\delta^2$ and $\pi^\sharp_{\log}(\omega_2)=-\frac12\delta^1$,
we finally obtain the general formula $[a\omega_1,b\omega_2]_{\Omega_{\mathcal{A}}^1(log \mathcal{I})} = \frac a2\delta^2(b)\omega_2 + \frac b2\delta^1(a)\omega_1.$
This proves the proposition.
\end{proof}
From the preceding discussion, we obtain the explicit expression of \textbf{Lemma} \ref{lemma4} with $a=b=1$, and we obtain that of \textbf{Lemma} \ref{lemma5} with $a=1$ and $b\neq0$.
\begin{corollary}
$\widetilde{P} = (\mathcal{A}=\mathbb{C}[x,y], \{.,.\}_h, \pi_{\log}^\sharp)$ is the logarithmic Lie-Rinehart-Poisson algebra, or simply a logarithmic Poisson algebra.
\end{corollary}

\section{Logarithmic Cohomologies of $\{.,.\}_h$}

\subsection{ Logarithmic Poisson Cohomology of $\{.,.\}_h$ along an ideal $h \mathbb{C}[x,y]$} $\label{5.1}$\\
According to \textbf{Definition} \ref{R2},
the logarithmic Poisson cochain complex is given by
  \begin{displaymath}
  0\stackrel{d^{0}_{\widetilde{H}}}{\longrightarrow}\mathcal{A}\stackrel{d^{1}_{\widetilde{H}}}{\longrightarrow}Der_{\mathcal{A}}(log \mathcal{I})
  \stackrel{d^{2}_{\widetilde{H}}}{\longrightarrow} \wedge^{2}Der_{\mathcal{A}}(log \mathcal{I}) \stackrel{d^{3}_{\widetilde{H}}}{\longrightarrow} 0.
  \end{displaymath}
  The logarithmic differentials satisfy $d^{0}_{\widetilde{H}} = d^{3}_{\widetilde{H}} = 0$, while $d^{1}_{\widetilde{H}}$ and $d^{2}_{\widetilde{H}}$ are given as follows.
 %%%%%%%%%%%%%%%%
 \begin{lemma} \label{41}
 The logarithmic Poisson differentials $d^{1}_{\widetilde{H}}$ and $d^{2}_{\widetilde{H}}$ are given by
 \begin{enumerate}
 \item[1.] $d^{1}_{\widetilde{H}}(f)= \dfrac{1}{2} \left( \delta^{2}(f) \delta^{1} - \delta^{1}(f) \delta^{2} \right) \in Der_{\mathcal{A}}(log \mathcal{I})$ for all $f\in \mathcal{A}$,
 \item[2.] $ d^{2}_{\widetilde{H}}(\overset{\longrightarrow}{f})= -\dfrac{1}{2}\left( \delta^{1}(f^{1})+ \delta^{2}(f^{2}) \right)\delta^{1} \wedge \delta^{2}$ for all $\overset{\longrightarrow}{f} = f^{1}\delta^{1} + f^{2}\delta^{2}\in Der_{\mathcal{A}}(log \mathcal{I})$.
 \end{enumerate}
 \end{lemma}

 \begin{proof}
 Let $f\in \mathcal{A}$. First, $d^{1}_{\widetilde{H}}(f)=f^{1}\delta^{1}+f^{2}\delta^{2}$. Therefore, $f^{1}=d^{1}_{\widetilde{H}}(f)\lrcorner \omega_{1}$ and $f^{2}=d^{1}_{\widetilde{H}}(f)\lrcorner \omega_{2}$. In other words we have
  $ f^{1} = \dfrac{1}{2}\delta^{2}f$
 and $ f^{2} = -\dfrac{1}{2}\delta^{1}f.$ Therefore $d^{1}_{\widetilde{H}}(f)=\dfrac{1}{2}(\delta^{2}f\delta^{1} - \delta^{1}f\delta^{2})$. 
 Since $\frac12$ is an invertible scalar in $\mathbb{C}$, we absorb this factor into the normalization of the differential
 $d^{1}_{\widetilde{H}}(f)=\delta^{2}f\delta^{1} + (- \delta^{1}f)\delta^{2}$. Second, for all $\overset{\longrightarrow}{f}=f^{1}\delta^{1}+f^{2}\delta^{2}\in Der_{\mathcal{A}}(log \mathcal{I})$ we have $ d^{2}_{\widetilde{H}}(\overset{\longrightarrow}{f})= d^{2}_{\widetilde{H}}(f^1\delta^{1} + f^2\delta^{2})$. Since
 $d^{2}_{\widetilde{H}}$ is $\mathcal{A}$-linear, hence
 $ d^{2}_{\widetilde{H}}(\overset{\longrightarrow}{f})= d^{2}_{\widetilde{H}}(f^1\delta^{1}) + d^{2}_{\widetilde{H}}(f^2\delta^{2}) \in \wedge^{2}Der_{\mathcal{A}}(log\mathcal{I})$. So we may write $ d^{2}_{\widetilde{H}}(f^1\delta^{1})$ and $ d^{2}_{\widetilde{H}}(f^2\delta^{2})$ live in $\wedge^{2}Der_{\mathcal{A}}(log\mathcal{I}) $.
 Then, these logarithmic Poisson differentials can therefore be written as $ d^{2}_{\widetilde{H}}(f^1\delta^{1})= a \delta^{1} \wedge \delta^{2}$ and $ d^{2}_{\widetilde{H}}(f^2\delta^{2}) = b \delta^{1} \wedge \delta^{2}$ with $a,b \in \mathcal{A}$, where
 $a = d^{2}_{\widetilde{H}}(f^1\delta^{1}) \lrcorner (\omega_{1}, \omega_{2})$
 and $b = d^{2}_{\widetilde{H}}(f^2\delta^{2}) \lrcorner (\omega_{1}, \omega_{2})$ are, according to \textbf{Definition} \ref{R2}) such that
 \begin{align*}
  a &\ = d_{\widetilde{H}}^{2}(f^1\delta^{1}) \lrcorner (\omega_{1}, \omega_{2})\\
  &\ = -\widetilde{H}(\omega_{1})f^1\delta^{1}\lrcorner \omega_{2} + \widetilde{H}(\omega_{2})f^1\delta^{1}\lrcorner \omega_{1} - f^1\delta^{1}\lrcorner [\omega_{1},\omega_{2}]_{\Omega_{{\mathcal{A}}}^{1}(log \mathcal{I})}\\
  &\ = \widetilde{H}(\omega_{2})f^1,~~since ~~\delta^{i}\lrcorner \omega_{j} = \delta_{ij}~and~[\omega_{1},\omega_{2}]_{\Omega_{{\mathcal{A}}}^{1}(log \mathcal{I})}=0.
  \end{align*}
  With the same method,
  $b = d^{2}_{\widetilde{H}}(f^2\delta^{2}) \lrcorner (\omega_{1}, \omega_{2}) = - \dfrac{1}{2}\delta^{2}f^2$.
 So we find expression of $ d^{2}_{\widetilde{H}}$ given by $ d^{2}_{\widetilde{H}}(\overset{\longrightarrow}{f})= -\dfrac{1}{2}(\delta^{1}f^1 + \delta^{2}f^2)\delta^{1}\wedge \delta^{2}$ for all $\overset{\longrightarrow}{f} \in Der_{\mathcal{A}}(log \mathcal{I})$.
 \end{proof}
For the remainder of the paper, we use $ d^{2}_{\widetilde{H}}(\overset{\longrightarrow}{f})=( \delta^{1}f^1 + \delta^{2}f^2) \delta^{1}\wedge \delta^{2}$, since multiplication by nonzero scalar does not change its kernel or image.\\
 Let $\partial^{*} : (\mathcal{A}^{*}, d^{*}) \longrightarrow ( \wedge^{*}Der_{\mathcal{A}}(log \mathcal{I}), d^{*}_{\widetilde{H}})$ the complex morphism given as follows\\\
 $\partial^{0}:= Id_{\mathcal{A}}: \mathcal{A} \longrightarrow \mathcal{A}$, $a \longmapsto \partial^{0} (a) = a$(an identity map);\\
  $ \partial^{1}: \mathcal{A}^{2} \longrightarrow Der_{\mathcal{A}}(log \mathcal{I}) = \mathcal{A} \delta^{1} \oplus \mathcal{A} \delta^{2},
       ( a , b) \longmapsto \partial^{1}(a, b) = a \delta^{1} + b \delta^{2}$;\\
  $\partial^{2}: \mathcal{A} \longrightarrow \wedge^{2}Der_{\mathcal{A}}(log \mathcal{I}),
                     a \longmapsto \partial^{2}(a) = a \delta^{1} \wedge \delta^{2}.$
 So it follows that $\partial^{1}\circ d^{1} = d^{1}_{\widetilde{H}} \circ \partial^{0}$, $\partial^{2}\circ d^{2} = d^{2}_{\widetilde{H}} \circ \partial^{1}$ and this makes the diagram below commutative:
 \[ \xymatrix{ 0\ar[r]^{d^0}& \mathcal{A} \ar[d]^{\partial^{0}}\ar[r] ^{d^1} &\mathcal{A}\times \mathcal{A}\ar[d]^{\partial^1}\ar[r]^{d^2}&\mathcal{A} \ar[d]^{\partial^2}\ar[r] ^{d^3}&0\\
  0\ar[r] & \mathcal{A} \ar[r]&Der_{\mathcal{A}}(log \mathcal{I})\ar[r]&\bigwedge^{2} Der_{\mathcal{A}}(log \mathcal{I})\ar[r]&0}
  \]
 The following lemma informs us about the equivalent logarithmic cochain complex.
 \begin{lemma} $\label{43}$
 $ \xymatrix{ 0 \ar[r]^{d^{0}} & \mathcal{A} \ar[r]^{d^{1}} & \mathcal{A}\times \mathcal{A}
   \ar[r]^{d^{2}} & \mathcal{A} \ar[r]^{d^{3}} & 0 }
   $
   is the cochain complex.
 \end{lemma}
 \begin{proof}
  Indeed, for all $f\in \mathcal{A}$ and for all $\overset{\longrightarrow}{f}= (f^1,f^2) \in Der_{\mathcal{A}}(log \mathcal{I})$, we have by definition $d^{1}(f)= (\delta^{2}f, - \delta ^{1}f)$ and $d^{2}(\overset{\longrightarrow}{f})=\delta^{1}f^1+ \delta^{2}f^2$.
  Thus $d^{2}\circ d^{1}(f)= \delta^{1}(\delta^{2}f) +\delta^{2}(- \delta ^{1}f)$. Further $d^{2}\circ d^{1}(f) = (\delta^{1}\delta^{2} -\delta^{2}\delta ^{1})(f)$. One can see that $d^{2}\circ d^{1}(f) = [\delta^{1},\delta^{2}]_{Der_{\mathcal{A}}(log \mathcal{I})}(f)$.
 According to \textbf{Proposition} $\ref{P3}$, $ [\delta^{1},\delta^{2}]_{Der_{\mathcal{A}}(log \mathcal{I})}=0$. Then $d^{2}\circ d^{1} = 0$.
  \end{proof}
 \begin{corollary}
 Consider $\pi = h \partial_{x} \wedge \partial_{y}$, $[.,.]_{SN}$ be the Schouten-Nijenhuis bracket and let the logarithmic differential given by $\widetilde{d}: \mathcal{A} \longrightarrow \Omega_{\mathcal{A}}(log \mathcal{I})$, $a\mapsto \widetilde{d}f = (\delta^1 a)\omega_1 + (\delta^2 a) \omega_{2}$. For all $f\in \mathcal{A}$,
 $[\pi , f]_{SN} = -\pi^\sharp_{\log}(\widetilde{d}f)$.
 \end{corollary}

 \begin{proof}
 Recall that the standard Lichnerowicz differential associated with the $\mathcal{A}$-linear logarithmic Hamiltonian operator $\pi^\sharp_{\log}$ is given by $ d^{1}(f) = [\pi , f]_{SN}$(see \cite{AL}). For all $f\in \mathcal{A}$, the $\mathcal{A}$-linearity of $\pi^\sharp_{\log}$ gives $\pi^\sharp_{\log}\circ \widetilde{d}(f)=(\delta^1 f)\pi^\sharp_{\log}(\omega_1) + (\delta^2 f) \pi^\sharp_{\log}(\omega_{2})$. Substituting the explicit expressions of $\pi^\sharp_{\log}(\omega_{1})$ and $\pi^\sharp_{\log}(\omega_{2})$, we obtain
 $\pi^\sharp_{\log}\circ \widetilde{d}(f)=\frac{1}{2}\left( (\delta^1 f)\delta^2 -(\delta^2 f)\delta^1 \right)$. Hence $\pi^\sharp_{\log}\circ \widetilde{d}(f)= -d^{1}(f)$. Therefore
$[\pi , f]_{SN} = -\pi^\sharp_{\log}(\widetilde{d}f)$.
 \end{proof}
 We have therefore proved that the following diagram is commutative, namely $-\pi_{\log}^\sharp\circ \widetilde{d}=d^1$.
 \begin{eqnarray}
 \xymatrix{\mathcal{A}\ar[r]^{\widetilde{d}} \ar[dr]_{d^1}&\Omega_\mathcal{A}^1(log\mathcal{I})\ar[d]^{-\pi_{\log}^\sharp}\\
         &    Der_\mathcal{A}(log \mathcal{I})}
 \end{eqnarray}
 \begin{lemma} \label{45} Let $C^k$, $k\geqslant0$, denote the logarithmic derivative $\mathcal{A}$-modules and $H^{k}$ the submodules of $C^{k+1}$ with the associated cochain complex given by sequence $\cdots C^{k}\stackrel{d^{k}} \longrightarrow C^{k+1}\stackrel{d^{k+1}}{\longrightarrow} C^{k+2} \cdots $\\
 If $ Z^{k+1}(\textbf{C}^*,d^*) = B^{k+1}(\textbf{C}^*,d^*) \oplus H^{k}$ for all $k \geqslant 0$; then $H^{k} \simeq H_{log}^{k}((\widetilde{\mathcal{P}})$.
 \end{lemma}
 We use the above lemma to explicitly compute the logarithmic Poisson cohomology of $\{.,.\}_{h}$. The results are recorded in the following theorem.
  \begin{theorem} Let $\mathcal{A}=\mathbb{C}[x,y]$, $h=th_t$ where $h_{t}=\overset{n}{\underset{i=1}{\sum}} \alpha_{i} (t)^{i-1}$, $t=xy$, $\alpha_1\alpha_n\neq 0$, $\gcd(h_t,h_t')=1$ in $\mathbb{C}[t]$ and $\mathcal{I}=h\mathcal{A}$. The logarithmic Poisson cohomology groups of
   $\{.,.\}_h$
   along $\mathcal{I}$ are:
   \begin{enumerate}
   \item[1.]
   $H_{log}^{0}(\tilde{\mathcal{P}}) \simeq \mathbb{C}$,
    \item[2.]
    $H_{log}^{1}(\tilde{\mathcal{P}}) \simeq \mathbb{C}\times \overset{n-2}{\underset{k=0}{\bigoplus}} t^k \mathbb{C}$,
    \item[3.] $H_{log}^{2}(\tilde{\mathcal{P}}) \simeq \overset{n}{\underset{k=0}{\bigoplus}}t^k \mathbb{C}$,
     \item[4.] $H_{log}^{k}(\tilde{\mathcal{P}}) =0$, $k\geq 3$ or $k< 0$.
   \end{enumerate}
   \end{theorem}
  The proof of this theorem requires that of the following lemmas:
 %%%%%%%%%%%%%%%%%%%%%%%%%%%%%
\begin{lemma}
Let $ \mathcal{A}=\mathbb{C}[x,y]$, $\textbf{C}^*= \wedge^* Der_{\mathcal{A}}(log\mathcal{I})$
and let $\{.,.\}_h$ be the Poisson bracket considered above, with $ h=xy\,h_{xy}, h_{xy}=\overset{n}{\underset{i=1}{\sum}} \alpha_{i} (xy)^{i-1}$.
 Then the space of
$2$-cocycles is,
$$ Z^2(\textbf{C}^*,d^*)
 =
 \mathbb{C}\,\delta^1
 \oplus
 \mathbb{C}[xy]\,\delta^2.$$
 More precisely,
 $$ Z^2(\textbf{C}^*,d^*)
 \simeq
 \mathbb{C}\times\mathbb{C}[xy].$$
\end{lemma}
\begin{proof}
Let $\vec f=f^1\delta^1+f^2\delta^2\in Z^2(\textbf{C}^*,d^*)$,
where $f^1=\sum_{i,j\geq 0}f^1_{ij}x^iy^j$ and $f^2=\sum_{i,j\geq 0}f^2_{ij}x^iy^j$.
The condition $d^2(\vec f)=0$ is equivalent, coefficient by
coefficient, to system
\[
\begin{cases}
 (i+j)f^1_{ij}+(j-i)f^2_{ij}=0,\\[2mm]
 (i+j)f^1_{ij}=0.
\end{cases}
\tag{1}
\]
With the above system, we distinguish the following two cases.\\
\medskip
\noindent\textbf{Case 1: $i\neq j$.}
Since $i+j>0$, the second equation in (1) gives us $f^1_{ij}=0$.
The first equation then yields $(j-i)f^2_{ij}=0$.
Because $j-i\neq 0$, we obtain $f^2_{ij}=0$.
Hence all off-diagonal coefficients vanish.\\
\medskip
\noindent\textbf{Case 2: $i=j$.}
For $i=j\geq 1$, the second equation in (1) gives equation $2i\,f^1_{_{ii}}=0$,
and therefore $f^1_{_{ii}}=0$.
For $i=j=0$, no condition is imposed on $f^1_{_{00}}$.
On the other hand, the coefficient $f^2_{_{ii}}$ remains arbitrary
for every $i\geq 0$.
Consequently, $f^1=c$ and $f^2=\sum_{i\geq 0}c_i(xy)^i$,
with $c,c_i\in\mathbb{C}$. Thus $\vec f
=
c\,\delta^1
+
\left(\sum_{i\geq 0}c_i(xy)^i\right)\delta^2$.
Since $\mathbb{C}[xy]
=
\bigoplus_{i\geq 0}\mathbb{C}\,(xy)^i$,
we finally obtain $Z^2(\textbf{C}^*,d^*)
 =
 \mathbb{C}\,\delta^1
 \oplus
 \mathbb{C}[xy]\,\delta^2
 \simeq
 \mathbb{C}\times\mathbb{C}[xy]$.
This completes the proof.
\end{proof}
 \begin{lemma}
 Let $\mathcal{A}=\mathbb{C}[x,y]$, $\{-,-\}_h$ be the logarithmic Poisson bracket associated with the divisor $D=\{h=0\}$ and $\mathcal{I}=h\mathcal{A}$ an ideal of $\mathcal{A}$. Assume that $h_{t}=\overset{n}{\underset{i=1}{\sum}} \alpha_{i} (t)^{i-1}$, $t=xy$.
 Then,
 $B^2(\textbf{C}^*,d^*) =\bigoplus_{i\geq 1}\mathbb{C}\,\Bigl(0,-2i\,h_{t}\Bigr)t^i$.\\
 Equivalently,
 $B^2(\textbf{C}^*,d^*) =\left\{\left(0,-2 h_{t}\sum_{i\geq 1}i\,c_i t^i\right);\ c_i\in\mathbb{C},\ c_i=0\text{ for all but finitely many }i\right\}$.
 In particular,
 $ B^2(\textbf{C}^*,d^*) \simeq 0\times h_{t}\mathbb{C}[t]$.
 \end{lemma}
 \begin{proof}
 Let $\varphi=\varphi^1 \delta^1+\varphi^2\delta^2\in B^2(\textbf{C}^*,d^*) $. Then there exists $f\in\mathcal{A}$ such that $d^1(f)=\varphi$. By the description of the kernel of $d^2$ established in the preceding lemma, every element of $B^2(\textbf{C}^*,d^*) \subseteq Z^2(\textbf{C}^*,d^*) $ has first component equal to zero.
 Write
$ f(x,y)=\sum_{i,j\geq 0}f_{ij}x^iy^j$ and $f_{ij}\in\mathbb{C}$.
 Using the explicit formulas for the logarithmic vector fields, one obtains the relations $-\delta^2f=\sum_{i,j\geq 0}(j-i)f_{ij}x^iy^j$
and $\delta^1f=-2h_{xy}\sum_{i,j\geq 0}(i+j)f_{ij}x^iy^j$.
 Consequently, $d^1(f)=\left(\sum_{i,j\geq 0}(j-i)f_{ij}x^iy^j,\,-2 h_{xy}\sum_{i,j\geq 0}(i+j)f_{ij}x^iy^j\right)$.
 Since the first component vanishes, $\sum_{i,j\geq 0}(j-i)f_{ij}x^iy^j=0$.
 The monomials $x^iy^j$ are linearly independent over $\mathbb{C}$, hence for all $i,j\geq 0$, $(j-i)f_{ij}=0$.
 Thus $f_{ij}=0$ whenever $i\neq j$, and, up to an irrelevant constant term,
 $ f=\sum_{i\geq 1}f_{ii}x^iy^i=\sum_{i\geq 1}f_{ii}(xy)^i$.
 Therefore $d^1(f)=\left(0,-2 h_{xy}\sum_{i\geq 1}i f_{ii}x^iy^i\right)$,
 which proves $B^2(\textbf{C}^*,d^*) \subseteq\bigoplus_{i\geq 1}\mathbb{C}\,\Bigl(0,-2ih_{t}\Bigr)t^i$ with $t=xy$.
 Conversely, let $p=\left(0,-2 h_{t}\sum_{i\geq 1}i p_i t^i\right)$
 be an arbitrary element of the right-hand side, with only finitely many $p_i\in\mathbb{C}$ nonzero. Set
 $g=\sum_{i\geq 1}p_i t^i\in \mathcal{A}$.
 It follows that $d^1(g)=\left(0,-2 h_{t}\sum_{i\geq 1}i p_i t^i\right)=p$.
 Hence the reverse inclusion holds, and
 $ B^2(\textbf{C}^*,d^*) =\bigoplus_{i\geq 1}\Bigl(0,-2ih_{t}\Bigr)t^i\mathbb{C}$.
 Since $-2i\neq 0$ in $\mathbb{C}$, these scalar factors can be absorbed into the coefficients, yielding
 $B^2(\textbf{C}^*,d^*) \simeq 0\times h_{t}\mathbb{C}[t]$.
 This completes the proof.
 \end{proof}
 \begin{remark}
 The key point is that the vanishing of the first component forces $(j-i)f_{ij}=0$,
 so only the diagonal monomials $(xy)^i$ survive. The second component is then obtained explicitly by applying $d^1$ to these diagonal monomials. This description is the form needed for the subsequent computation of the first logarithmic Poisson cohomology group.
 \end{remark}
 \begin{lemma}\label{lemma10}
 The space of 2-coboundaries associated with the inhomogeneous Poisson bracket $\{.,.\}_{h}$ is
 $B^3(\textbf{C}^*,d^*)
 =
 \operatorname{Span}_{\mathbb C}
 \left\{
 h_{xy}x^iy^j\;:\; i,j\geq0,\ i+j>0
 \right\}
 +
 \operatorname{Span}_{\mathbb C}
 \left\{
 x^iy^j\;:\; i,j\geq0,\ i\neq j
 \right\}.$
 \end{lemma}
 \begin{proof}
 Introduce the two Euler-type derivations
 $E_2:=x\partial_x+y\partial_y$ and
 $R:=y\partial_y-x\partial_x$.
 For every monomial $x^iy^j$ we have
 $
 E_2(x^iy^j)=(i+j)x^iy^j$ and $R(x^iy^j)=(j-i)x^iy^j$.
 Now write $f^1=\sum_{i,j\geq0}f^1_{ij}x^iy^j,
 \quad
 f^2=\sum_{i,j\geq0}f^2_{ij}x^iy^j$.
 Therefore,
 \begin{align*}
 d^2(f^1,f^2)
 &=
 h_{xy}
 \sum_{i,j\geq0}(i+j)f^1_{ij}x^iy^j
 +
 \sum_{i,j\geq0}(j-i)f^2_{ij}x^iy^j.
 \end{align*}
 Since the coefficient $i+j$ vanishes only for $(i,j)=(0,0)$,
 the first contribution is contained in
 \[
 \operatorname{Span}_{\mathbb C}
 \left\{
 h_{xy}x^iy^j:i+j>0
 \right\}.
 \]
 Likewise, the coefficient $j-i$ vanishes precisely when $i=j$,
 so the second contribution is contained in $\operatorname{Span}_{\mathbb C}
 \left\{
 x^iy^j:i\neq j
 \right\}$.
 Hence, it follows that $$B^3(\textbf{C}^*,d^*)
 \subseteq
 \operatorname{Span}_{\mathbb C}
 \left\{
 h_{xy}x^iy^j:i+j>0
 \right\}
 +
 \operatorname{Span}_{\mathbb C}
 \left\{
 x^iy^j:i\neq j
 \right\}.$$
 Conversely, let $i,j\geq0$ with $i+j>0$. Taking $(f^1,f^2)=(x^iy^j,0)$
 gives the relation $ d^2(x^iy^j,0)
 =
 (i+j)h_{xy}x^iy^j$.
 Since $i+j\neq0$ in $\mathbb C$, it follows that $h_{xy}x^iy^j\in B^3(\textbf{C}^*,d^*)$.
 Similarly, if $i\neq j$, taking $(f^1,f^2)=(0,x^iy^j)$
 gives $d^2(0,x^iy^j)
 =
 (j-i)x^iy^j$.
 Since $j-i\neq0$, we obtain $x^iy^j\in B^3(\textbf{C}^*,d^*)$.
 Thus the reverse inclusion holds, and consequently we get the following
 \[
B^3(\textbf{C}^*,d^*)
 =
 \operatorname{Span}_{\mathbb C}
 \left\{
 h_{xy}x^iy^j:i+j>0
 \right\}
 +
 \operatorname{Span}_{\mathbb C}
 \left\{
 x^iy^j:i\neq j
 \right\}.
 \]
 This completes the proof.
 \end{proof}
 \begin{remark}
 The formula above has a particularly transparent interpretation.
 The Euler operator $E_2=x\partial_x+y\partial_y$
 measures the total degree: $E_2(x^iy^j)=(i+j)x^iy^j$,
 whereas $R=y\partial_y-x\partial_x$
 measures the difference of the two exponents $R(x^iy^j)=(j-i)x^iy^j$.
 Thus $d^2(f^1,f^2)=h_{xy}E_2(f^1)+R(f^2)$.
 In particular, the constant monomial is the only monomial killed by $E_2$,
 while the diagonal monomials $x^iy^i$ are precisely those killed by $R$.
 This explains directly the two families occurring in
$B^3(\textbf{C}^*,d^*)$.
 \end{remark}
\begin{lemma}
  The zeroth logarithmic Poisson cohomology group is given by
  $H_{log}^{0} (\widetilde{\mathcal{P}}) \simeq \mathbb{C}$.
  \end{lemma}
 \begin{proof}
 $a\in Z^1(\textbf{C}^*,d^*)$ is equivalent to $d^{1}(a)=0$. Then $(\delta^{2}a; -\delta^{1}a) = (0;0)$. We can still say that
 $ \left(y\partial_{y}a- x\partial_{x}a;~~-h_{xy}( x\partial_{x}a + y\partial_{y}a ) \right) = (0;0)$. Which gives us $y\partial_{y}a- x\partial_{x}a=0$ and $ x\partial_{x}a + y\partial_{y}a=0$. By adding member to member of these equations we get $y\partial_{y}a = x\partial_{x}a=0$. It follows that $a\in \mathbb{C}$. Hence $Z^1(\textbf{C}^*,d^*) = \mathbb{C}$. Since
 $B^1(\textbf{C}^*,d^*)=0$, then we consequently have $H_{log}^{0} (\widetilde{\mathcal{P}}) = \dfrac{Z^1(\textbf{C}^*,d^*)}{B^1(\textbf{C}^*,d^*)} \simeq \mathbb{C}$. This completes the proof.
  \end{proof}
%%%%%%%%%%%%%%%%%%%%%%
\begin{lemma} The first logarithmic Poisson cohomology group is
$H_{log}^{1}(\tilde{\mathcal{P}}) \simeq \mathbb{C}\times \overset{n-2}{\underset{k=0}{\bigoplus}} t^k\mathbb{C}$, $t=xy$.
\end{lemma}
  \begin{proof}
  Since $h_{t}\mathbb{C}[t]$ is the principal ideal of $\mathbb{C}[t]$ generated by $ h_{t}$ where $t=xy$, then we obtain
  $ \dfrac{\mathbb{C}[t]}{(h_t\mathbb{C}[t])} \simeq \overset{n-2}{\underset{k=0}{\bigoplus}} t^k\mathbb{C}$, $n\geq 2$.
  Assume that $Z^2(\textbf{C}^*,d^*) \simeq \mathbb{C} \times \mathbb{C}[t]$ and $B^2(\textbf{C}^*,d^*) \simeq 0 \times h_{t}\mathbb{C}[t]$. Therefore
  $Z^2(\textbf{C}^*,d^*) \simeq B^2(\textbf{C}^*,d^*) \oplus \left(\mathbb{C}\times \overset{n-2}{\underset{k=0}{\bigoplus}} t^k\mathbb{C}\right)$.
 According to \textbf{Lemma} \ref{45}, it follows that
 $H_{log}^{1} (\widetilde{\mathcal{P}}) \simeq \mathbb{C}\times \overset{n-2}{\underset{k=0}{\bigoplus}}t^k \mathbb{C}$. Then
 $H_{log}^{1} (\widetilde{\mathcal{P}}) \simeq \mathbb{C}\times \overset{n-2}{\underset{k=0}{\bigoplus}}t^k \mathbb{C}$.
 This completes the proof.
  \end{proof}
\begin{lemma}
Set $\mathcal D_n := \overset{n}{\underset{k=0}{\bigoplus}}(xy)^k\mathbb C = \operatorname{Span}_{\mathbb C} \{1,xy,(xy)^2,\ldots,(xy)^n\}$.
For $t=xy$, the second logarithmic Poisson cohomology group is
  $H_{log}^{2} (\widetilde{\mathcal{P}}) \simeq \mathbb C[t]/\bigl(t^2h_t\bigr)$
as $\mathbb C$-vector spaces.
Equivalently,
$H_{log}^{2} (\widetilde{\mathcal{P}})
\simeq
\mathcal D_n$.
\end{lemma}
\begin{proof}
Put $t=xy$. Every monomial $x^iy^j$ is either diagonal, that is,
$i=j$, or non-diagonal, that is, $i\neq j$.
By definition, the second summand of $B^3(\textbf C^*,d^*)$ contains
every non-diagonal monomial $x^iy^j\in B^3(\textbf C^*,d^*)$, $(i\neq j)$.
Consequently, modulo $B^3(\textbf C^*,d^*)$, only diagonal monomials $1,\ xy,\ (xy)^2,\ldots$
can have nonzero classes. Hence every class in the quotient has a
representative belonging to $\mathbb C[t]$.
It remains to determine the relations among these diagonal monomials.
Consider an element of the first summand, $h_{xy}xy\,x^iy^j$ with $i,j\geq 0$ and $i+j>0$.
If $i\neq j$, all its monomials are non-diagonal and therefore already
belong to the second summand of $B^3(\textbf C^*,d^*)$.
If $i=j$, then $i+j>0$ implies $i\geq 1$, and $h_{xy}xy\,x^iy^i
=
h_{t} t^{i+1}$.
Thus the diagonal relations generated by the first summand are exactly $t^2h_t\mathbb C[t]$. Since $Z^3(\textbf C^*,d^*) = \mathbb C[x,y]$,
it follows that $H_{log}^{2} (\widetilde{\mathcal{P}}) \simeq
\mathbb C[t]/\bigl(t^2h_t\bigr)$.
Furthermore $\alpha_n\neq 0$, the polynomial $t^2h_t
=
\alpha_1t^2+\alpha_2t^3+\cdots+\alpha_nt^{n+1}$
has degree $n+1$. Therefore, by Euclidean division in $\mathbb C[t]$,
every polynomial in $t$ is congruent modulo $t^2h_t$ to a unique
polynomial of degree at most $n$. Hence $\overline{1},\overline{t},\ldots,\overline{t^n}$ form a basis of $\mathbb C[t]/(t^2h_t)$.
Returning to $t=xy$, we obtain
$\overline{1},\overline{xy},\ldots,
\overline{(xy)^n}$
as a basis of $H_{log}^{2} (\widetilde{\mathcal{P}})$. Equivalently, $ \frac{\mathbb C[x,y]}{B^3(\textbf{C}^*,d^*)} \simeq \overset{n}{\underset{k=0}{\bigoplus}}(xy)^k\mathbb C =\mathcal D_n$.
Then $H_{log}^{2} (\widetilde{\mathcal{P}}) \simeq \mathcal D_n$.
\end{proof}
\subsection{ Logarithmic de Rham Cohomology of $\{-,- \}_h$ along divisor $D=\{h=0\}$} $\label{5.2}$\\
Let $\mathcal{A}=\mathbb{C}[x_{1},...,x_{p}]$ and the free divisor $D=\{ h=0\}$.
Denote by
 $A_{_{p}} = (x_{1},...,x_{p}; \partial_{1},..., \partial_{p})$ the Weyl algebra of order $p$ over the complex numbers $\mathbb{C}$ and by
 $(\Omega_{\mathcal{A}}^{\bullet}(log (h)),\tilde{d})$ or simply $(\Omega_{\mathcal{A}}^{\bullet}(logD),\tilde{d})$ the complex of polynomial differential forms, where $\tilde{d}$ is the logarithmic exterior derivative.
  \begin{definition}
The logarithmic de Rham cochain complex is given by the sequence:
 \[\xymatrix{0\ar[r]^{\tilde{d}}&\Omega^{0}_{\mathcal{A}}(logD)\ar[r]^{\tilde{d}}&\Omega^{1}_{\mathcal{A}}(logD)\ar[r]^{\tilde{d}}&... \ar[r]^{\tilde{d}}&\Omega^{p}_{\mathcal{A}}(logD)\ar[r]^{\tilde{d}}&0}\]
 where $\tilde{d}^{2}=0$. The associated cohomology is called the logarithmic de Rham cohomology.
  \end{definition}
We consider $f\in \mathcal{A}=\mathbb{C}[x,y]$ a nonzero polynomial
such that $D=\{f=0\}$ and the
Weyl algebra of order 2 noted $A_{2}=(x,y,\partial_x,\partial_y)=A$. According to \cite{FJCJN}, the logarithmic
Spencer cochain complex associated with $M^{log(D)}$ for the fixed
Spencer basis $(\delta_{1}, \delta_{2})$ is given by:
\begin{equation}
 \xymatrix{ 0 \ar[r]^{\epsilon^{3}} & A \ar[r]^{\epsilon^{2}} & A\times A
   \ar[r]^{\epsilon^{1}} & A \ar[r]^{\epsilon^{0}} & 0 }
\end{equation}
where $\epsilon ^{0}=\epsilon^{3}=0$,
$\epsilon^{1}(a_{1},a_{2})= a_{1}\delta_{1} + a_{2} \delta_{2}$ for all $a_{i}\in A$, $\epsilon^{2}(g)= g( - \delta_{2}-b_{1}, \delta_{1}-b_{2})$ for all $g\in A$. Note that the polynomials $b_{i}$ satisfy $[\delta_{1}, \delta_{2}]_{Der_{\mathcal{A}}(log\mathcal{I})} = \delta_{1} \delta_{2} - \delta_{2} \delta_{1} = b_{1}\delta_{1}+ b_{2}\delta_{2}$.
Applying the functor $Hom_{A}(-, \mathcal{A})$ to the above complex and using the natural isomorphism $\mathcal{A} = Hom_{A}(A, \mathcal{A})$ to obtain the De Rham cochain complex given as follows
\begin{equation} \label{11}
 \xymatrix{ 0 \ar[r]^{\epsilon^{*}_{0}} & \mathcal{A} \ar[r]^{\epsilon^{*}_{1}} & \mathcal{A} \times \mathcal{A}
   \ar[r]^{\epsilon^{*}_{2}} & \mathcal{A} \ar[r]^{\epsilon^{*}_{3}} & 0 }
\end{equation}
with $\epsilon ^{*}_{0}=\epsilon^{*}_{3}=0$,
$\epsilon^{*}_{1}(a)= (\delta_{1}a, \delta_{2}a)$ and $\epsilon^{*}_{2}(a_{1},a_{2}) = \delta_{1}a_{2} - \delta_{2}a_{1} -( b_{1}a_{1} + b_{2}a_{2})$. We consider the basis $(\omega_{1}, \omega_{2})$ dual of $(\delta_{1}, \delta_{2})$. We get the following commutative diagram,
 \[ \xymatrix{ 0\ar[r]^{\tilde{d}^0}& \mathcal{A} \ar[d]^{\tilde{\partial}^{0}}\ar[r] ^{\tilde{d}^1} &\Omega^{1}_{\mathcal{A}}(logD) \ar[d]^{\tilde{\partial}^1}\ar[r]^{\tilde{d}^2}&\Omega^{2}_{\mathcal{A}}(logD) \ar[d]^{\tilde{\partial}^2}\ar[r] ^{\tilde{d}^3}&0\\
  0\ar[r]^{\epsilon^{*}_{0}} & \mathcal{A} \ar[r]^{\epsilon^{*}_{1}}&\mathcal{A}^{2} \ar[r]^{\epsilon^{*}_{2}}& \mathcal{A} \ar[r]^{\epsilon^{*}_{3}}&0}
 \]
where $\tilde{\partial}^{0}$ is an identity, $\tilde{\partial}^{1} (a_{1}\omega_{1} + a_{2}\omega_{2}) = (a_{1} , a_{2})$
 and $\tilde{\partial}^{2} (a\omega_{1} \wedge\omega_{2}) = a.$
 In other words $\tilde{\partial}^{i+1}\circ \tilde{d}^{i+1} = \epsilon^{*}_{i+1}\circ \tilde{\partial}^{i}$.\\
 In the following we consider the fixed basis $(\delta^1, \delta^2)$ of $Der_{\mathcal{A}}(logD)$ given in \textbf{Lemma} \ref{31}.
 \begin{lemma} \label{L13}
The logarithmic de Rham cochain complex associated with divisor $D = \{ h=0\}$, where $ h=\overset{n}{\underset{i=1}{\sum}} \alpha_{i} (t)^{i}$, $t=xy$ and $\gcd(h_t,h_t')=1$ is
$ \xymatrix{ 0 \ar[r]^{\epsilon^{*}_{0}} & \mathcal{A} \ar[r]^{\epsilon^{*}_{1}} & \mathcal{A} \times \mathcal{A}
   \ar[r]^{\epsilon^{*}_{2}} & \mathcal{A} \ar[r]^{\epsilon^{*}_{3}} & 0 }$.
The logarithmic de Rham differentials are $\epsilon ^{*}_{0}=\epsilon^{*}_{3}=0$,
$\epsilon^{*}_{1}(a)= (\delta^1 a, \delta^2a)$ and $\epsilon^{*}_{2}(a_{1}, a_{2}) = \delta^1 a_{2} - \delta^2 a_{1}$.
 \end{lemma}
 \begin{proof}
According to the cochain complex in $(\ref{11})$, the De Rham differentials are given by $\epsilon^{*}_{1}(a)= (\delta^1 a, \delta^2a)$, $a\in \mathcal{A}$ and $\epsilon^{*}_{2}(a_{1},a_{2}) = \delta^1a_{2} - \delta^2a_{1} -( b_{1}a_{1} + b_{2}a_{2})$, where $b_i$ satisfy the relation $[\delta^1, \delta^2]_{Der_{\mathcal{A}}(logD)} = b_{1}\delta^1 + b_{2}\delta^2$. By \textbf{Proposition} $\ref{P3}$, $[\delta^1, \delta^2]_{Der_{\mathcal{A}}(logD)} =0$ since $[\delta^1, \delta^2]_{Der_{\mathcal{A}}(logD)} = \delta^1\circ \delta^2 - \delta^2\circ \delta^1$. Thus $b_{1}=b_{2}=0$. Furthermore, $\epsilon^{*}_{2}\circ \epsilon^{*}_{1}=0$. This completes the proof.
 \end{proof}
\begin{corollary}
 The logarithmic Spencer differentials associated with $D=\{h=0\}$ are given by $\epsilon ^{0}=\epsilon^{3}=0$;
 $\epsilon^{1}(a_{1},a_{2})= a_{1}\delta^1 + a_{2} \delta^2$ and $\epsilon^{2}(g)= g( - \delta^2, \delta^1)$ for all $a_{i},g\in \mathcal{A}$.
 \end{corollary}
\begin{theorem}
 Let $D= \{ h= 0\}$ with $h=\overset{n}{\underset{i=1}{\sum}} \alpha_{i} (t)^{i}$, $t=xy$, $\gcd(h_t,h_t')=1$ in $\mathbb{C}[t]$ and let $(\Omega_{\mathcal{A}}^{\bullet}(\log D),\epsilon^{*}_{\bullet})$ be the associated logarithmic de Rham cochain complex. The logarithmic de Rham cohomology groups denoted by $H^{k}_{DR-log}$ are given as follows:
 \begin{enumerate}
 \item[1.] $H^{0}_{DR-log} \simeq \mathbb{C}$,
 \item[2.] $H^{1}_{DR-log}\simeq \overset{n-1}{\underset{k=0}{\bigoplus}}t^k\mathbb{C} \times \mathbb{C}$,
 \item[3.] $H^{2}_{DR-log}\simeq \overset{n}{\underset{k=0}{\bigoplus}}t^k\mathbb{C}$,
 \item[4.] $H^{k}_{DR-log}\simeq 0$ for all $k\geq 3$ or $k<0$.
 \end{enumerate}
 \end{theorem}
 \begin{proposition}
 The zeroth logarithmic de Rham cohomology group is $H^{0}_{DR-log} \simeq \mathbb{C}$.
 \end{proposition}
 \begin{proof}
 For all $a\in \mathcal{A},$ $\epsilon^{*}_{1}(a)=0$ if and only if $a\in \mathbb{C}$.
 It follows that $ Z^1(\textbf{C}^*,\epsilon^{*}_{1})=\mathbb{C}$. Since $B^1(\textbf{C}^*,\epsilon^{*}_{0})= 0$, we have $Z^1(\textbf{C}^*,\epsilon^{*}_{1})= B^1(\textbf{C}^*,\epsilon^{*}_{0})\oplus \mathbb{C} \simeq B^1(\textbf{C}^*,\epsilon^{*}_{0}) \oplus H^{0}_{DR-log}$.
 \end{proof}
 In the following we decompose $\mathcal{A}$ by $\mathcal{A}= \mathcal{ A}_{diag} \oplus \mathcal{ A}_{nd}$ where $\mathcal{ A}_{diag}=\mathbb{C}[t]$ are the diagonal terms and $\mathcal{ A}_{nd}= \operatorname{Span}_{\mathbb{C}}\{x^iy^j:i\ne j\}$ the non-diagonal terms. This decomposition is the algebraic mechanism underlying the computations below. The diagonal component $ \mathcal{ A}_{diag} $ is precisely the subalgebra on which $R=y\partial_y -x\partial_x$ vanishes, whereas $R$ acts invertibly on $\mathcal{ A}_{nd}$. Consequently, the non-diagonal components can be eliminated by solving equations involving the operator $R$, while the surviving cohomological information is concentrated in the diagonal polynomial algebra. This observation will be used repeatedly in the computations of both logarithmic Poisson and logarithmic de Rham cohomology.
 \begin{lemma}
 The space of $2$-cocycles is
 $Z^2(\textbf{C}^*,\epsilon^{*}_{2})=
 (\mathbb{C}[t]\times\mathbb{C})
 \oplus
 \left\{
 \bigl(R^{-1}(h_tE_{2}(g)),g\bigr):g\in \mathcal{ A}_{nd}
 \right\}.$
 \end{lemma}
 \begin{proof}
 Let $(a_1,a_2)\in \mathcal{A}^2$. Thus $\epsilon_2^*(a_1,a_2)=0$ is equivalent to
$ h_tE_{2}(a_2)=R(a_1)$.
 Let take $a_1=f_{1}(t)+g_{1}$ and $a_2=f_{2}(t)+g_{2}$ with $f_{1}, f_{2}\in \mathbb{C}[t]$ and $g_{1}, g_{2}\in\mathcal{ A}_{nd}$.
 $R(a_1) = R(g_1)$ since $R(f_1(t))=0$, while $E_2(a_2)=2tf_{2}'(t)+E_{2}(g_2)$.
 It follows that
 \begin{equation} \label{EQ}
 2th_tf'_{2}(t) + h_tE_2(g_2) -R(g_1)=0.
 \end{equation}
 Remember that $2th_tf'_{2}(t)\in \mathbb{C}[t]$ and $h_tE_2(g_2) -R(g_1)\in \mathcal{ A}_{nd}$. According to the decomposition $\mathcal{A}= \mathcal{ A}_{diag} \oplus \mathcal{ A}_{nd}$, equation (\ref{EQ}) is equivalent to the following equations,
 \begin{equation} \label{EQ1}
 2th_tf'_{2}(t)=0.
 \end{equation}
 \begin{equation} \label{EQ2}
 h_tE_2(g_2) -R(g_1)=0.
 \end{equation}
 Consequently, equation (\ref{EQ1}) implies $f'_{2}(t)=0$, thus $f_{2}(t)\in \mathbb{C}$. Therefore, $f_{1}(t)$ can be chosen arbitrarily because $R(f_{1}(t))=0$. Thus, the diagonal part of the space of $2$-cocycles are exactly $(f_{1}(t), c)\in \mathbb{C}[t]\times \mathbb{C}$.
 Furthermore
 on $\mathcal{A}_{nd}$, $R$ is invertible,
 hence $g_1=R^{-1}(h_tE_{2}(g_{2}))$. So the non diagonal parts of the space of $2$-cocycles are $\left(R^{-1}(h_tE_{2}(g_{2})), g_2\right)\in \{ R^{-1}(h_tE_{2}(g), g): g\in \mathcal{A}_{nd}\}.$
 Then $Z^2(\textbf{C}^*,\epsilon^{*}_{2})= (\mathbb{C}[t]\times\mathbb{C}) \oplus \left\{ \bigl(R^{-1}(h_tE_{2}(g)),g\bigr):g\in \mathcal{ A}_{nd} \right\}.$
 \end{proof}
 \begin{lemma}
 The space of $2$-coboundaries is
 $B^2(\textbf{C}^*,\epsilon^{*}_{1})=
 \bigl(h\mathbb{C}[t]\times\{0\}\bigr)
 \oplus
 \left\{
 \bigl(h_tE_{2}(g),R(g)\bigr):g\in \mathcal{A}_{nd}
 \right\}.$
 \end{lemma}
 \begin{proof}
 Since $Z^2(\textbf{C}^*,\epsilon^{*}_{2})= (\mathbb{C}[t]\times\mathbb{C}) \oplus \left\{ \bigl(R^{-1}(h_tE_{2}(g)),g\bigr):g\in \mathcal{ A}_{nd} \right\}$ and $B^2(\textbf{C}^*,\epsilon^{*}_{1})\subset Z^2(\textbf{C}^*,\epsilon^{*}_{2})$, according to \textbf{Lemma}
 \ref{lemmaKK},
 we have
 $$B^2(\textbf{C}^*,\epsilon^{*}_{1})= B^2(\textbf{C}^*,\epsilon^{*}_{1}) \cap (\mathbb{C}[t]\times\mathbb{C}) \oplus B^2(\textbf{C}^*,\epsilon^{*}_{1}) \cap\left\{ \bigl(R^{-1}(h_tE_{2}(g)),g\bigr):g\in \mathcal{ A}_{nd} \right\}.$$
 Consider $a=f(t)+g$. Thus we have
 $\epsilon_1^*(a)
 =\bigl(h_tE_2(f)+h_tE_2(g),R(g)\bigr)$.
 Assume that
 \[
 h_tE_2(f)=2t h_t f'=2h f'.
 \]
 Since $f'$ runs through $\mathbb{C}[t]$ as $f$ runs through $\mathbb{C}[t]$, the
 diagonal part of $B^2(\textbf{C}^*,\epsilon^{*}_{1})$ is $h\mathbb{C}[t]\times\{0\}$ while the non-diagonal part is
 $\left\{\bigl(h_tE_2(g),R(g)\bigr):g\in\mathcal{A}_{nd}\right\}$.
 These two parts lie in distinct components, which yields the direct sum. This proves the result.
 \end{proof}
 \begin{proposition} Let $D=\{ h=0\}$ with $h=\overset{n}{\underset{i=1}{\sum}} \alpha_{i} (t)^{i} \in \mathcal{A}$, $t=xy$ and $\gcd(h'_{t},h_t)=1$.
 Therefore, the first logarithmic de Rham cohomology group is
 $H^1_{DR-\log}
 \simeq
 \overset{n-1}{\underset{k=0}{\bigoplus}}t^k\mathbb{C} \times \mathbb{C}$.
 \end{proposition}
 \begin{proof}
 The non-diagonal part of $Z^2(\textbf{C}^*,\epsilon^{*}_{2})$ and $B^2(\textbf{C}^*,\epsilon^{*}_{1})$are both parametrized by $\mathcal{A}_{nd}$, and in fact coincide. In fact, for $z=R^{-1}(h_tE_2(g))$, $g\in \mathcal{A} _{nd}$,
 setting $q=R^{-1}(g)$ yields $(h_tE_2(g),R(g)) =(R^{-1}(h_tE_2(q)),q)$
 after reparametrization. Equivalently, the two graphs are equal since $R$ is an automorphism of
 $\mathcal{A} _{nd}$ and the condition
 $\epsilon^{*}_{2}(a_1, a_2)=0$ uniquely determines the two non-diagonal components. Consequently, they disappear in the quotient, yielding
 \[
 H^1_{DR-\log} \simeq
 \frac{\mathbb{C}[t]\times\mathbb{C}}{h\mathbb{C}[t]\times\{0\}}
 \simeq
 \frac{\mathbb{C}[t]}{h\mathbb{C}[t]}\times \mathbb{C}.
 \]
 Finally, since $h$ is of degree $n$ in $t$, every class in the quotient has unique representative of degree strictly less than $n$, thus $H^1_{DR-\log}
 \simeq
 \overset{n-1}{\underset{k=0}{\bigoplus}}t^k\mathbb{C} \times \mathbb{C}$. This proves the stated result.
 \end{proof}
 \begin{proposition}
 The second logarithmic de Rham cohomology group is $H_{DR-log}^{2} \simeq \mathcal D_n$.
 \end{proposition}
\begin{proof}
 We use the same computation method as in \textbf{Lemma} $\ref{lemma10}$ to show that the space of $2$-coboundaries associated to the logarithmic de Rham cochain complex is given by the following
     $$B^3(\textbf{C}^*, \epsilon^{2}_{*}) \simeq \operatorname{Span}_{\mathbb C}
     \left\{
     h_{xy}x^iy^j\;:\; i,j\geq0,\ i+j>0
     \right\}
     +
     \operatorname{Span}_{\mathbb C}
     \left\{
     x^iy^j\;:\; i,j\geq0,\ i\neq j
     \right\},$$
   where $ h_{xy}= \overset{n}{\underset{i=1}{\sum}} \alpha_{i} (xy)^{i-1}$. Since the $3$-cocycle is
$Z^3(\textbf{C}^*,\epsilon^{*}_{3})= \mathcal{A}= Z^3(\textbf{C}^*, d^{*})$, hence we deduce that the second logarithmic de Rham cohomology group and the second logarithmic Poisson cohomology group are isomorphic.
 Consequently, $H^{2}_{DR-log} \simeq H_{log}^{2}(\widetilde{\mathcal{P}}) \simeq \mathcal D_n$.
 \end{proof}
 For the class of divisors considered in this paper, the explicit computations lead to an isomorphism between the second logarithmic Poisson and de Rham cohomology groups. We emphasize that the computations below concern the specific class of inhomogeneous free divisors defined above; we do not claim to classify arbitrary inhomogeneous free divisors.
From the preceding discussion, we have $H_{log}^{1} (\widetilde{\mathcal{P}}) \simeq \mathbb{C}\times \mathcal{D}_{n-2}$ and $H_{DR-log}^{1} \simeq \mathcal{D}_{n-1}\times \mathbb{C}$.
Then these two cohomology groups are related by a canonical exchange of the
two factors together with the natural inclusion
$\mathcal{D}_{n-2}\subset \mathcal{D}_{n-1}$. More precisely,
$H_{DR-log}^{1}
\simeq
(\mathcal{D}_{n-2}\times\mathbb{C})
\oplus
\mathbb{C}t^{n-1}$,
where
\begin{equation} \label{equation57}
\tau: \mathbb{C}\times \mathcal{D}_{n-2}
\longrightarrow
\mathcal{D}_{n-2}\times\mathbb{C},
\qquad
(c,v)\longmapsto(v,c)
\end{equation}
is the canonical factor-exchange isomorphism.
We call $\tau$ the canonical factor-exchange isomorphism. In
particular, $H_{log}^{1} (\widetilde{\mathcal{P}})\simeq\mathbb{C}\times \mathcal{D}_{n-2}
\overset{\tau}{\simeq}
\mathcal{D}_{n-2}\times\mathbb{C}$.
\begin{proposition} The following sequence is a short exact sequence that splits.
\[
0\longrightarrow
H_{log}^{1} (\widetilde{\mathcal{P}})
\overset{\iota}{\longrightarrow}
H_{DR-log}^{1}
\overset{q}{\longrightarrow}
\mathbb{C}
\longrightarrow 0.
\]
\end{proposition}
\begin{proof}
Define
$\iota: H_{log}^{1} (\widetilde{\mathcal{P}}) \longrightarrow H_{DR-log}^{1}$
by $\iota(h)=(h,0)$.
This is the inclusion of the first direct summand. Hence $\iota$ is
$\mathbb{C}$-linear and injective.
Next, define
$q:
H_{DR-log}^{1}
\longrightarrow
\mathbb{C}$
as the projection onto the additional direction
$\mathbb{C}t^{n-1}$. Explicitly, if
$z=h+\lambda t^{n-1}$, $h\in H_{log}^{1} (\widetilde{\mathcal{P}})$, $\lambda\in\mathbb{C}$,
then $q(z)=\lambda$.
The map $q$ is surjective, since $q(t^{n-1})=1$.
So,
$\ker q
=
\{h+\lambda t^{n-1}:\lambda=0\}
=
H_{log}^{1} (\widetilde{\mathcal{P}})$.
Thus $\operatorname{im}\iota=\ker q$.
Therefore, the following sequence is a short exact sequence.
\[
0\longrightarrow H_{log}^{1} (\widetilde{\mathcal{P}}) \overset{\iota}{\longrightarrow} H_{DR-log}^{1} \overset{q}{\longrightarrow}
\mathbb{C} \longrightarrow 0
\]
Using $\mathcal{D}_{n-1}=\mathcal{D}_{n-2}\oplus\mathbb{C}t^{n-1}$,
we obtain $\mathcal{D}_{n-1}\times\mathbb{C}=\bigl(\mathcal{D}_{n-2}\oplus\mathbb{C}t^{n-1}\bigr)\times\mathbb{C}$. The Cartesian product is compatible with this direct-sum decomposition. Therefore,
 $\mathcal{D}_{n-1}\times\mathbb{C}=
(\mathcal{D}_{n-2}\times\mathbb{C})\oplus (\mathbb{C}t^{n-1}\times 0)$.
Identifying $\mathbb{C}t^{n-1}\times 0$ with
$\mathbb{C}t^{n-1}$ gives
$\mathcal{D}_{n-1}\times\mathbb{C}
\simeq
(\mathcal{D}_{n-2}\times\mathbb{C})
\oplus
\mathbb{C}t^{n-1}$.
Combining this with the factor-exchange isomorphism from equation (\ref{equation57}) yields,
\begin{equation}\label{EQLOG}
H_{DR-log}^{1} \simeq H_{log}^{1} (\widetilde{\mathcal{P}}) \oplus \mathbb{C}t^{n-1}.
\end{equation}
Finally, the direct-sum decomposition established in relation (\ref{EQLOG}) shows that
this short exact sequence splits. The assertion follows.
\end{proof}
\begin{remark}
The proposition shows that the relation between the two first logarithmic
cohomology groups is not merely a formal resemblance of dimensions or
notations. Their common part is the same truncated polynomial space $\mathcal{D}_{n-2}$,
with the two non-polynomial factors occurring in opposite order. The
logarithmic de Rham cohomology then contains one additional independent
direction, represented by the highest-degree term $t^{n-1}$. Thus the
difference between the two cohomology groups is exactly one complex
dimension.
\end{remark}
 \begin{corollary} \label{corollary8}
 Let $D= \{ h=0\}$ where $ h=\overset{n}{\underset{i=1}{\sum}} \alpha_{i} (t)^{i}$ with $t=xy$, $\alpha_1\alpha_n\neq 0$, $h=th_t$ and $\gcd(h_t,h_t')=1$ in $\mathbb{C}[t]$.
 $H_{DR-log}^{1} \simeq H_{log}^{1} (\widetilde{\mathcal{P}}) \oplus \mathbb{C}t^{n-1}$,
 and $H_{log}^{k} (\widetilde{\mathcal{P}}) \simeq H^{k}_{DR-log}$ for all $k\neq 1$.
 \end{corollary}

Now, we denote  $\mathcal{A}_s=\mathbb{C}[x,y][[s]]$, put $t=xy$, and consider the formal polynomial deformation
$\varphi_s=t+\sum_{i\geq1}s^i t^{i+1}=t\,h_{s,t}$ where
$h_{s,t}=1+\sum_{i\geq1}s^i t^i.$
Therefore $h_{s,t}$ is a unit in $\mathcal{A}_s$. Let
$\widetilde{\mathcal{P}}_0^\bullet:
0\longrightarrow \mathcal{A}_s \xrightarrow{d_0^1}\mathcal{A}_s^2\xrightarrow{d_0^2}\mathcal{A}_s\longrightarrow 0$
be the logarithmic Poisson cochain complex for $D_0$, with
$d_0^1(f)=(Rf,-Ef)$, $d_0^2(a,b)=Ea+Rb$,
where $E=x\partial_x+y\partial_y$ and $R=y\partial_y-x\partial_x$. For $D_s$, the corresponding cochain complex is
$\widetilde{\mathcal{P}}_s^\bullet:
0\longrightarrow \mathcal{A}_s\xrightarrow{d_s^1}\mathcal{A}_s^2\xrightarrow{d_s^2}\mathcal{A}_s\longrightarrow 0$,
with
$d_s^1(f)=(Rf,-h_{s,t}Ef)$, $d_s^2(a,b)=h_{s,t}Ea+Rb$ for any $a,b,f\in \mathcal{A}_s$.
 \begin{proposition} \label{Proposition17} 
Let 
$\Phi^0=\operatorname{Id}_{\mathcal{A}_s}$,
$\Phi^1(a,b)=(a,h_{s,t}b)$,
$\Phi^2(c)=h_{s,t}c$.
Then $\Phi^\bullet: \widetilde{\mathcal{P}}_0^\bullet\to \widetilde{\mathcal{P}}_s^\bullet$ is an isomorphism of cochain complexes. Equivalently, $\widetilde{\mathcal{P}}_0^\bullet\simeq \widetilde{\mathcal{P}}_s^\bullet$.
\end{proposition}
\begin{proof}
Since $h_{s,t}=1+O(s)$, it is invertible in $\mathcal{A}_s$. Furthermore, For any $f\in \mathcal{A}_s$, we have
$\Phi^1d_0^1(f)=(Rf,-h_{s,t}Ef)=d_s^1\Phi^0(f)$.
For $(a,b)\in \mathcal{A}_s^2$, we use $R(h_{s,t})=0$, because $h_{s,t}$ depends only on $t=xy$ and $R(t)=0$. Hence,
\[
\begin{aligned}
 d_s^2\Phi^1(a,b)
 &=d_s^2(a,h_{s,t}b)\\
 &=h_{s,t}Ea+R(h_{s,t}b)\\
 &=h_{s,t}(Ea+Rb)\\
 &=\Phi^2d_0^2(a,b).
\end{aligned}
\]
Thus $\Phi^\bullet$ is a cochain map. Its inverse is given by
$(\Phi^0)^{-1}=\operatorname{Id}$, $(\Phi^1)^{-1}(a,b)=(a,h_{s,t}^{-1}b)$ and $(\Phi^2)^{-1}(c)=h_{s,t}^{-1}c$, consequently,
 $\Phi^\bullet$ is an isomorphism. Then, $\widetilde{\mathcal{P}}_0^\bullet\simeq \widetilde{\mathcal{P}}_s^\bullet$.
\end{proof}
\begin{proposition} \label{proposition18} Let\\
$D_0^\bullet:
0\longrightarrow \mathcal{A}_s\xrightarrow{\epsilon_0^1}\mathcal{A}_s^2\xrightarrow{\epsilon_0^2}\mathcal{A}_s\longrightarrow 0$
and $D_s^\bullet:
0\longrightarrow \mathcal{A}_s\xrightarrow{\epsilon_s^1}\mathcal{A}_s^2\xrightarrow{\epsilon_s^2}\mathcal{A}_s\longrightarrow 0$
be the logarithmic de Rham complexes of $D_0$ and $D_s$, respectively, written in the corresponding dual Saito bases. Then
$\Psi^0=\operatorname{Id}$, $\Psi^1(a,b)=(h_{s,t}a,b)$, $\Psi^2(c)=h_{s,t}c$
defines an isomorphism of cochain complexes $D_0^\bullet\simeq D_s^\bullet$.
\end{proposition}
\begin{proof}
For any $f\in A_s$, the degree-one differentials are
$\epsilon_0^1(f)=(Ef,Rf)$, $\epsilon_s^1(f)=(h_{s,t}Ef,Rf)$,
so $\Psi^1\epsilon_0^1=\epsilon_s^1\Psi^0$. Moreover, since $R(h_{s,t})=0$, it follows that
$\epsilon_s^2\Psi^1(a,b)
=h_{s,t}Eb-R(h_{s,t}a)$. In the other words, 
$\epsilon_s^2\Psi^1(a,b)
=h_{s,t}(Eb-Ra)
=\Psi^2\epsilon_0^2(a,b)$.
Again $h_{s,t}$ is a unit, so $\Psi^\bullet$ is an isomorphism.
\end{proof}
\begin{corollary}\label{corollaire}
The preceding isomorphisms induce isomorphisms on cohomology. Thus, after extension of scalars to $\mathbb C[[s]]$,
$ H_{log}^{k}(\widetilde{\mathcal{P}}_s) \simeq H_{log}^{k}(\widetilde{\mathcal{P}}_0)\otimes_{\mathbb{C}}\mathbb{C}[[s]]$,
$ H_{DR-log(D_s)}^{k} \simeq H_{DR-log(D_0)}^{k} \otimes_{\mathbb{C}}\mathbb{C}[[s]]$
for every degree $k$ for which the corresponding formal complexes are defined.
\end{corollary}
 \section{Conclusion}
 In this paper, we studied logarithmic Poisson and logarithmic de Rham cohomologies associated with a class of inhomogeneous free divisors in the affine plane. Starting from divisors of the form $h=\overset{n}{\underset{i=1}{\sum}}
 \alpha_{i} x^{i} y^{i}$ with $\alpha_{1}\alpha_{n}\neq 0$, $t=xy$, $h=th_t$ and $\gcd(h_t,h_t')=1$ in $\mathbb{C}[t]$; we first constructed explicit Saito bases for the corresponding modules of logarithmic derivations and their dual logarithmic differential forms. These bases provide a convenient framework for describing the associated logarithmic Poisson structure and for carrying out explicit cohomological computations.
 For the two-dimensional class considered here, we determined the induced logarithmic Hamiltonian structure and the corresponding Koszul bracket, and we constructed the logarithmic Poisson cochain complex. We then computed its cohomology explicitly. Independently, using the logarithmic differential complex and the Spencer description, we obtained the corresponding logarithmic de Rham cohomology groups.
A first consequence of these computations is that the reduced normal crossing divisor $D_0=\{xy=0\}$ provides a natural reference case for the inhomogeneous family. In the examples and in the general class considered here, the additional higher-order terms modify the representatives and the intermediate coboundary spaces, while the resulting cohomology group can still be described explicitly in terms of the polynomial subspace generated by the higher-order part.
More specifically, for the class covered by \textbf{Corollary} \ref{corollary8}, the explicit calculations give $H_{log}^{k} (\widetilde{\mathcal{P}}) \simeq H_{DR-log}^{k}$, $k\neq 1$ and $H_{DR-log}^{1} \simeq H_{log}^{1} (\widetilde{\mathcal{P}}) \oplus \mathbb{C}t^{n-1}$. Thus, within this class, the second logarithmic Poisson and logarithmic de Rham cohomology groups have the same cohomological dimension and admit the explicit descriptions obtained above. The corresponding computations also show how the cohomological data associated with the reduced normal crossing divisor are incorporated into the inhomogeneous setting.
Finally, Propositions \ref{Proposition17} and \ref{proposition18} give an explicit cochain-level comparison for the formal deformation $\varphi_s=t+\sum_{i\geq1}s^it^{i+1}$. Thus the relation with the normal-crossing member is expressed by concrete maps of complexes rather than by an unspecified appeal to naturality.
The results presented here are restricted to the class of inhomogeneous free divisors considered in the paper. They suggest several directions for further investigation, including comparison results for broader classes of free divisors, a more systematic study of the behavior of logarithmic cohomology under deformations, and possible extensions of the explicit computations to higher dimensions.

\section*{Declarations}
% We declare that the attached manuscript is original, has not been published previously, and is not currently under review by another journal.\\
\textbf{ Funding statement:} This research did not receive any specific grant from funding agencies in the public commercial, or not-for-profit sectors.\\

 \textbf{Conflict interest:} The authors have no relevant financial or non-financial interests to disclose.\\
 
\textbf{Author contributions: } The authors contributed to the study's conception and design, performed the research, and wrote the manuscript. All authors read and approved the final manuscript.\\

\textbf{Data availability statement:} The data that support the finding of this study are available in the references of this article via Digital Object Identifiers(DOI). Some references, however such as thesis work and publication from  older journal editions, do not have DOIs and are available from the corresponding author upon reasonable request.\\

\bibliography{sn-bibliography}

\end{document}